\documentclass[oneside,reqno, 12pt]{amsart}

\usepackage{geometry}
\usepackage{amsmath}
\usepackage{amsfonts}
\usepackage{amsthm}
\usepackage{mathtools}
\usepackage{amssymb}
\usepackage{commath}
\usepackage{marginnote}

\usepackage{enumerate}

\usepackage[all]{xy}
\usepackage{tikz}
\usetikzlibrary{cd}
\usetikzlibrary{positioning,babel,shapes.geometric}

\usepackage{caption}
\usepackage{subcaption}

\usepackage[square,numbers]{natbib}

\theoremstyle{plain}
\newtheorem{thm}{Theorem}[section]
\newtheorem{lem}[thm]{Lemma}
\newtheorem{prop}[thm]{Proposition}
\newtheorem{cor}[thm]{Corollary}

\theoremstyle{definition}
\newtheorem{defn}[thm]{Definition}

\newtheorem{exmp}[thm]{Example}

\theoremstyle{remark}
\newtheorem*{rem}{Remark}

\newcommand{\imagekernel}{\texttt{Image\_Kernel}}

\newcommand{\bF}{\mathbb{F}}

\newcommand{\bN}{\mathbb{N}}

\newcommand{\bR}{\mathbb{R}}

\newcommand{\bX}{\mathbb{X}}
\newcommand{\bY}{\mathbb{Y}}
\newcommand{\bZ}{\mathbb{Z}}

\newcommand{\cA}{\mathcal{A}}
\newcommand{\cB}{\mathcal{B}}
\newcommand{\cC}{\mathcal{C}}
\newcommand{\cD}{\mathcal{D}}
\newcommand{\cE}{\mathcal{E}}
\newcommand{\cF}{\mathcal{F}}

\newcommand{\cH}{\mathcal{H}}
\newcommand{\cI}{\mathcal{I}}
\newcommand{\cJ}{\mathcal{J}}

\newcommand{\cL}{\mathcal{L}}

\newcommand{\cP}{\mathcal{P}}
\newcommand{\cQ}{\mathcal{Q}}
\newcommand{\cR}{\mathcal{R}}

\newcommand{\cU}{\mathcal{U}}
\newcommand{\cV}{\mathcal{V}}
\newcommand{\cW}{\mathcal{W}}

\newcommand{\bbR}{{\rm \bf R}}

\newcommand{\Ho}{{\rm H}}
\newcommand{\PH}{{\rm PH}}
\newcommand{\VR}{{\rm VR}}

\newcommand{\Tot}{{\rm Tot}}

\newcommand{\SpSq}{{\rm \bf SpSq}}
\newcommand{\PSpSq}{{\rm \bf PSpSq}}

\newcommand{\veps}{\varepsilon}

\newcommand{\PMod}{\rm \bf PMod}

\newcommand{\vect}{{\rm \bf vect}}

\newcommand{\Id}{{\rm Id}}
\newcommand{\Ker}{{\rm Ker}}
\newcommand{\Img}{{\rm Im}}

\newcommand{\Coker}{{\rm Coker}}

\newcommand{\dimn}{{\rm dim}}
\newcommand{\dimnp}[1]{{\rm dim}(#1)}

\newcommand{\cech}{{\rm \v{C}ech}}

\newcommand{\cell}{{\rm cell}}

\newcommand{\colim}{{\rm \bf colim \,}}
\newcommand{\hocolim}{{\rm \bf hocolim \,}}
\newcommand{\Bd}{{\rm \bf Bd \,}}

\newcommand{\CWcpx}{\mbox{\bf CW-cpx}}
\newcommand{\FCWcpx}{\mbox{\bf FCW-cpx}}
\newcommand{\RCWcpx}{\mbox{\bf RCW-cpx}}
\newcommand{\RDiag}[1]{{\rm \bf RDiag}(#1)}
\newcommand{\RRDiag}[1]{{\rm \bf RRDiag}(#1)}
\newcommand{\FRDiag}[1]{{\rm \bf FRDiag}(#1)}
\newcommand{\FFDiag}[1]{{\rm \bf FFDiag}(#1)}

\DeclareMathOperator{\img}{Im}

\newcommand{\MN}{{\rm MNerv}}

\title{Interleaving Mayer-Vietoris spectral sequences}
\author{\' Alvaro Torras} \thanks{\' Alvaro Torras research has been supported by an EPSRC grant reference EP/N509449/1 number 1941653.}
\email{TorrasCasasA@cardiff.ac.uk}

\author{Ulrich Pennig}
\email{pennigu@cardiff.ac.uk}



\begin{document}

\maketitle

\begin{abstract}
We discuss the Mayer-Vietoris spectral sequence as an invariant in the context of persistent homology. In particular, we introduce the notion of $\veps$-acyclic carriers and $\veps$-acyclic equivalences between filtered regular CW-complexes and study stability conditions for the associated spectral sequences. We also look at the Mayer-Vietoris blowup complex and the geometric realization, finding stability properties under compatible noise; as a result we prove a version of an approximate nerve theorem. Adapting work by Serre we find conditions under which $\veps$-interleavings exist between the spectral sequences associated to two different covers.
\end{abstract}

\section{Introduction}

One of the benefits of homology as a topological invariant, over for example the
homotopy groups, is its computability via long exact sequences. The classical
Mayer-Vietoris exact sequence has been used in countless examples to compute
$\Ho_k(X)$ from a decomposition of $X$ into two open subsets $U$ and $V$.
When we generalise this concept to open covers $(U_i)_{i \in I}$ consisting of
more than just two subsets, the relations between the parts $\Ho_k(U_i)$ become
more intricate and are encoded in the Mayer-Vietoris spectral sequence.
These sequences first appeared in work of Leray and later Serre, and they
proved to be one of the most powerful tools in pure algebraic topology.
Applications of spectral sequences in applied algebraic topology, however, are
still a young subject.

In~\citep{TorrasCasas2019} it was proven that the Persistence
Mayer-Vietoris spectral sequence can be used to compute persistent homology.
The starting point is a filtered simplicial complex $X$ together with a
cover by subcomplexes $\cU$. Then, one computes $\PH_i(\cU_\sigma)$ for all
$i\geq 0$ and $\sigma \in N_\cU$, where $N_\cU$ denotes the nerve of the cover $\cU$. The Mayer-Vietoris spectral sequence
starts from these groups and the morphisms induced by inclusions and converges to $\PH_i(X)$.
As pointed out in~\citep{YooGhr2020} the additional insight gained from the cover $\cU$ can be used for example for multiscale feature detection. Similar information was also explored much earlier in~\citep{ZomCar2008} in the form of \emph{localized homology}.

Motivated by these results one might be interested in studying the spectral sequence
$E^*_{p,q}(X,\cU)$ as an invariant in its own right. In this paper we will in particular pursue the following questions:
\begin{itemize}
  \item Given a pair $(X,\cU)$ consisting of a space~$X$ and a cover $\cU$, an object closely related to the Mayer-Vietoris spectral sequence $E^*_{p,q}(X, \cU)$ is the Mayer-Vietoris blowup complex $\Delta_\cU(X)$. Is $\Delta_\cU(X)$ stable? In which way?

  \item Suppose that the data in each covering set $\cU_\sigma$ for $\sigma \in N_\cU$ is modified slightly. If the underlying cover $\cU$ is ignored, then we would not expect $E^*_{p,q}(X, \cU)$ to be stable. Are there natural coherence conditions between changes in the sets $\cU_\sigma$ that imply stability? If so, what do we mean by stability of spectral sequences?

  \item Let $\cU$ and $\cV$ be covers of the same space $X$. Can we compare $E^*_{p,q}(X, \cU)$ and $E^*_{p,q}(X, \cV)$ up to $\veps$-interleavings?
\end{itemize}

To explain why the first question is important and how it is linked to spectral sequences, we note that $E^*_{p,q}(X, \cU)$ \emph{converges} to the \emph{target} persistent homology $\PH_*(\Delta_\cU(X))$ (this is usually denoted by $E^*_{p,q}(X, \cU) \Rightarrow \PH_*(\Delta_\cU(X))$). The blowup complex $\Delta_\cU(X)$ already appeared in the context of topological data analysis in~\citep{LewisMorozov2015} and~\citep{ZomCar2008}. It is homotopy equivalent to a \emph{homotopy colimit}, and therefore enjoys good properties with respect to local homotopy equivalences. For example, if we assume that
$X(\sigma) := X \cap \cU_\sigma$ is contractible for all $\sigma \in N^\cU$,
then we can use~\citep[Prop.~4G.2]{Hatcher2002} to recover
Leray's Nerve Theorem. That is, there are homotopy equivalences
$$
X \simeq \Delta_\cU(X) \simeq \Delta_\cU(*) = N(\cU),
$$
where $*$ denotes the constant \emph{complex of spaces} on $\cU$, see~\citep[App.~4.G]{Hatcher2002}.
The fundamental importance of this result in applied topology
is underlined by the persistent Nerve lemma presented in~\citep{ChazOudo2008}.
It is worth mentioning the Approximate Nerve Theorem~\citep{GovSkr2018} and the
Generalized Nerve Theorem~\citep{Cavanna2019}, which are approximate versions of
the Leray Theorem within the context of persistence. In particular,
in~\citep{GovSkr2018} the spectral sequence $E^*_{p,q}(X,\cU) \Rightarrow \PH_*(X)$
is examined, and it is studied how much it differs from the other spectral
sequence $E^*_{p,q}(*,\cU) \Rightarrow \PH_*(N(\cU))$, by careful inspection of
all pages as well as the extension problem.

Throughout the paper we focus on the category $\RCWcpx$ of
\emph{regularly filtered regular CW complexes} as well as the subcategory $\FCWcpx$
of \emph{filtered regular CW complexes}, see subsection~\ref{sub:FRcpx}. Instead of
restricting our attention to a space $X$ together with a cover $\cU$, we look at
regular diagrams $\cD$ in $\RCWcpx$ over a simplicial complex~$K$. There is a
natural replacement for the Mayer-Vietoris blowup complex in this setting, denoted
by $\Delta_K(\cD)$, as explained in~\citep[App.~4.G.]{Hatcher2002}. This object
also appears in the context of semisimplicial spaces, where it is called the
\emph{geometric realization}~\citep{EbeRan2019}; in fact, it has an associated
spectral sequence~\citep[Sub.~1.4.]{EbeRan2019}. As we explain in
Sec.~\ref{sec:spectral_sequences} there are good reasons as to why it is worth
taking this more general perspective. In particular, we consider the spectral sequence
\[
E^2_{p,q}(\cD)\Rightarrow \PH_{p+q}(\Delta_K\cD)\ .
\]

In order to address the first two questions, we introduce the notion of acyclic
carriers to define \emph{$\veps$-acyclic equivalences}. Using the
Acyclic Carrier Theorem we show the following: Let $X$ and $Y$ be two objects
in $\RCWcpx$. If there exists an $\veps$-acyclic equivalence
$F^\veps:X\rightrightarrows Y$, then $\PH_*(X)$ is $\veps$-interleaved with
$\PH_*(Y)$ (see Lem.~\ref{lem:RCWcpx-veps-interleaving} and Prop.~\ref{prop:eps-acyclic-carriers} for a stronger statement in $\FCWcpx$).
These equivalences provide a very flexible notion that works in different contexts as the
examples~\ref{ex:vietoris-rips-complexes} and~\ref{ex:cubical-complexes} show.
Then we adress the first question in the following way.  Let $\cD$ and $\cL$ be
two diagrams over the same simplicial complex $K$ and assume that for all $\sigma \in K$
there are $\veps$-acyclic equivalences $F^\veps_\sigma:\cD(\sigma)\rightrightarrows \cL(\sigma)$
which satisfy a compatibility condition with respect to composition in the poset category
associated to $K$, see proposition
(see Prop.~\ref{thm:geometric-realization} for details), then there is an $\veps$-acyclic
equivalence $F^\veps:\Delta_K(\cD)\rightrightarrows \Delta_K(\cL)$. This result implies
stability in the targets of convergence of the spectral sequences. We use this result to
show a `Strong Approximate Multinerve Theorem' in
Thm.~\ref{thm:strong-multinerve}.
Later in section~\ref{sec:ss-interleavings}, we introduce $(\veps, n)$-interleavings,
which are given by spectral sequence morphisms that start at some page $n$ together with
a shift by a persistence parameter $\veps>0$. Assuming the same conditions as in the Geometric Realization
case~(Prop.~\ref{thm:geometric-realization})
we can obtain a $(\veps, 1)$-interleaving
between $E_{p,q}^*(\cD)$ and $E_{p,q}^*(\cL)$. This result appears in Theorem~\ref{thm:stability-ss} and
a specialized strong statement for covers of spaces in $\FCWcpx$ is given in Proposition~\ref{prop:stability-ss}.

As for the third question about the comparison of the spectral sequences
associated to two covers $\cU$ and $\cV$ of a space $X$, we rely on work of Serre
from the fifties, in which he studied the relation between the
\cech~cohomology of two different covers~\citep{Serre1955}; here we adapt this
work in the context of cosheaves and cosheaf homology. Given a cosheaf $\cF$
of abelian groups on $X$ and assume that there is a refinement $\cV \prec \cU$.
Serre showed that the refinement morphism induced on \v{C}ech~homology
$\rho^{\cU\cV}:\check{\cH}_*(\cV, \cF) \rightarrow \check{\cH}_*(\cU, \cF)$ is
independent of the particular choice of morphism in the cochains.
In~\citep{Serre1955} it was also shown that $\rho^{\cU\cV}$ can be factored
through a construction that uses a double complex associated
to both covers $C_{p,q}(\cU, \cV;\cF)$, see \citep[Prop.~4, Sec.~29]{Serre1955}.
This construction introduces two double complex spectral sequences
${^{\rm I}E}_{p,q}^*(\cU, \cV;\cF)$ and ${^{\rm II}E}_{p,q}^*(\cU, \cV;\cF)$, both
of which converge to $\check{\cH}_*(\cU\cap \cV; \cF) \simeq \check{\cH}_*(\cV;\cF)$.
Here one might study conditions on ${^{\rm II}E}_{p,q}^*(\cU, \cV;\cF)$ to find
when an inverse of $\rho^{\cU\cV}$ exists.
As an application, Serre obtained an analogous result to the Leray
Theorem in the context of sheaves \citep[Thm.~1 in \S 29]{Serre1955}.

We start our analysis of the third question in
Sec.~\ref{sec:interleavings-different-covers}. In case $\cV \prec \cU$ there
is a unique morphism induced by the refinement map on the second page
\[
\rho^{\cU\cV}:E^*_{p,q}(X,\cV)\rightarrow E^*_{p,q}(X,\cU)\ .
\]
On the other hand,
Thm.~\ref{thm:inverse-refinement} tells us under what conditions there exists a
$\veps$-shifted morphism $\psi:E^*_{p,q}(X,\cU) \rightarrow
 E^*_{p,q}(X,\cV)[\veps]$ so that $\rho^{\cU\cV}$ and $\psi$ form an
$(\veps,2)$-interleaving between $E^*_{p,q}(X,\cU)$ and $E^*_{p,q}(X,\cV)$.
Finally, in Prop.~\ref{prop:local-checks} we give a means of obtaining
an $(\veps,2)$-interleaving between $E^*_{p,q}(X,\cU)$ and $E^*_{p,q}(X,\cV)$
through the computation of local Mayer-Vietoris spectral sequences
$E^*_{p,q}(\cU_\sigma, \cV_{|\cU_\sigma})$ for all $\sigma \in N_\cU$. Since
the open regions $\cU_\sigma$ are assumed to be `small' in comparison to
$X$, this gives a means of using local calculations to deduce the interleaving.
As Corollary~\ref{cor:stability-covers} we present the case when $\cV$ does
not need to refine $\cU$.

\section*{Acknowledgements}

We would like to thank P.~Skraba and D.~Govc for fruitful discussions during
spring 2020 that lead up to important ideas of this manuscript. In particular
P.~Skraba pointed out to us the 1955 notes~\citep{Serre1955} from
J.~P.~Serre, which have been key for the results from section~\ref{sec:interleavings-different-covers}.

\section{Background}

\subsection{Regular CW-complexes with filtrations}
\label{sub:FRcpx}

Recall the definition of CW-complex from~\citep[Chapter 0]{Hatcher2002}. In contrast to
the usual treatment of CW-complexes, but in line with the structure we are dealing with in TDA, we will consider the cell decomposition as part of the data of our CW-complexes.
For a CW-complex $X$, if $c$ is an open cell in $X$ we will follow the notation
from~\cite{CookeFinney1967} and denote this by $c \in X$. We will denote by
$X^n$ the set of $n$-dimensional cells from $X$ and we will denote by
$X^{\leq n}$ the $n$-skeleton from $X$. Recall that $X$ has a natural filtration
given by its skeleta $X^{0} \subseteq  X^{\leq 1} \subseteq \cdots
\subseteq X^{\leq N} \subseteq \cdots$, and a
\emph{cellular morphism} $f:X \to Y$ respects this filtration, in the sense
that it restricts to morphisms $f^m: X^{\leq m}\to Y^{\leq m}$ for all $m \geq 0$.
We will work with regular CW-complexes, which are CW-complexes where the
attaching maps are homeomorphisms.  Given a pair of cells $a\in X^n$ and
$b \in X^{n-1}$, we denote by $[b : a]$ the degree of attaching map $\partial a \to \overline{b}/\partial b$.
%from $a$ to $b$.
A cellular morphism $f:X \to Y$ will be called a \emph{regular morphism}
whenever the closure  $\overline{f(a)}$ is a subcomplex of $Y$ for all cells
$a \in X$. For such a morphism and and a pair $a \in X^n$ and $b \in Y^n$, we
will denote by $[b:f(a)]$ the degree of the morphism $f$ restricted to the open
cell~$a$ and mapping into the open cell $b$.
We will write $\CWcpx$ to denote the category of finite regular CW-complexes and
regular morphisms.
Denote by $\bbR$ the ordered category $(\bR, \leq)$
of real numbers. We will focus on functors $X: \bbR \rightarrow \CWcpx$ which we
will call \emph{regularly filtered CW complexes}, and we denote their
category by $\RCWcpx$. We say that an object $X \in \RCWcpx$ is
\emph{tame}, whenever $X$ is constant along a finite
number of right open intervals decomposing the poset $\bbR$. For $X \in \RCWcpx$,
we will write $X_r$ for the regular CW-complex $X(r)$ for all $r \in \bbR$.
On the other hand we will write $X(r\leq s)$ to denote the morphism
$X_r \to X_s$ for all $r\leq s$ in $\bbR$. If the morphisms
$X(r\leq s):X_r \to X_s$ are injections preserving the cellular structure for
all $r\leq s$ in $\bbR$, then we call $X$ a \emph{filtered CW-complex},
denoting by $\FCWcpx$ the corresponding subcategory of $\RCWcpx$.
 Notice that objects
in \FCWcpx~can be seen as a pair $(\colim X_*, f)$ where $\colim X_*$ is a
regular CW-complex and $f:\colim X_*\rightarrow \bR$ is a filtration function.
Given $X \in \RCWcpx$, we define the persistent homology in degree $n$ as the
functor $\PH_n(X):\bbR \rightarrow \vect$ given by computing cellular homology
$\PH_n(X)_r = \Ho_n^\cell(X_r)$ for all $r \in \bbR$. We will always use homology with field coefficients. Notice that by finiteness
of $X_r$, the vector space
$\PH_n(X)_r$ is finite dimensional for all $r\in \bbR$. If in addition $X$ is tame,
 $\PH_n(X)$ only changes at a finite number
of points $r\in \bbR$. We call the category of functors
$\bbR\rightarrow \vect$ \emph{persistence modules} and
denote it by $\PMod$.
Given $a \in (0, \infty)$ and $X \in \RCWcpx$, we will write $X[a]$ for the
element of $\RCWcpx$ such
that $X[a]_r = X_{r+a}$ for all $r \in \bbR$.
We use $\Sigma^\veps$ to denote the $\veps$-\emph{shift functor}
$\Sigma^\veps: \RCWcpx \rightarrow {\rm Hom}(\RCWcpx)$ which sends
$X \in \RCWcpx$ to $\Sigma^\veps X: X \rightarrow X[\veps]$, where
$\veps \geq 0$. Also, for any morphism
of filtered CW-complexes $f:A\rightarrow B$, one can check that
$f [\veps]\circ \Sigma^\veps A = \Sigma^\veps B \circ f$, where
we use $f [\veps] : A[\veps]\rightarrow B[\veps]$.
Similarly, there are shift functors for persistence modules
$\Sigma^\veps : \PMod \rightarrow {\rm Hom}(\PMod)$ for $\veps \geq 0$.

\begin{rem}
Notice that the standard algorithm for the computation of persistent homology can not be applied to objects in $\RCWcpx$.
However, if $X$ is tame and one successfully computes the coefficients
for the morphisms $C^\cell_*(X_r)\to C^\cell_*(X_s)$ for all $r \leq s$ in $\bbR$,
then one can use~\imagekernel~from~\citep{TorrasCasas2019} to obtain a
\emph{barcode basis} for the filtered cellular complex $C^\cell_*(X)$. Then we
compute homology of the persistence morphisms given by the differentials
$d_n:C^\cell_n(X)\rightarrow C^\cell_{n-1}(X)$ by the use of~\imagekernel.
See~\citep{TorrasCasas2019} for an explanation.
\end{rem}

\subsection{Acyclic carriers}
Consider two objects $\Phi$ and $\Gamma$ from $\CWcpx$ with
their respective pairs of chains and differentials
$\big(C^\cell_*(\Phi), \delta^\Phi\big)$ and
$\big(C^\cell_*(\Gamma), \delta^\Gamma\big)$.
Let $\langle \cdot, \cdot \rangle_\Phi$ and
$\langle \cdot, \cdot \rangle_\Gamma$
denote the inner products on $C^\cell_*(\Phi)$ and $C^\cell_*(\Gamma)$,
where the cells form an orthonormal basis.
We define a relation $\prec$ on $\Phi$ by setting
$\tau \prec \sigma$ if
$\langle \tau, \delta^\Phi(\sigma) \rangle_\Phi \neq 0$ and by taking the transitive closure. We will denote by
$\preceq$ the partial order generated by $\prec$. Thus,
$\tau \prec \sigma$ will not necessarily imply $\dimnp{\tau} + 1 = \dimnp{\sigma}$.
Also, notice that $\langle \tau, \delta^\Phi(\sigma) \rangle_\Phi=[\tau:\sigma]$,
see the cellular boundary formula from~\citep[Sec. 2.2]{Hatcher2002}.
\begin{defn}
A \emph{carrier} $F:\Phi \rightrightarrows \Gamma$ is a map from the set of
cells of $\Phi$ to subcomplexes of $\Gamma$ that is semicontinuous in the sense
that for any pair $\tau \prec \sigma$ in $\Phi$, $F(\tau)\subseteq F(\sigma)$.
A carrier $F:\Phi \rightrightarrows \Gamma$ is called \emph{acyclic}, if for
every $\sigma \in \Phi$, $F(\sigma)$ is a nonempty acyclic subcomplex of
$\Gamma$.
\end{defn}

Given a chain map $w_p:C_p^\cell{}(\Phi) \rightarrow C_{p+r}^\cell(\Gamma)$ of degree $r$,
we say that it is carried by $F$ if for all cells $\sigma \in \Phi_p$
$$
\{ \gamma \in \Gamma_{p+r} \mid \langle w_p(\sigma),\gamma \rangle_\Gamma \neq 0\}\subseteq F(\sigma)\ ,
$$
where we followed the notation from~\citep{Nanda2012}.

The next statement is an application of~\citep[Thm.~13.4]{Munk1984}.
In Proposition~\ref{prop:eps-acyclic-carriers} we will prove
a version of this statement that will apply to filtered CW-complexes.
\begin{thm}
\label{thm:acyclic-carrier}
Let $F:\Phi\rightrightarrows \Gamma$ be an acyclic carrier between CW-complexes
$\Phi$ and~$\Gamma$. Then we have that
\begin{itemize}
\item {\bf existence:} there is a chain map carried by $F$,
\item {\bf equivalence:} if $F$ carries two chain maps $\phi$ and $\varphi$, then $F$ carries
a chain homotopy between $\phi$ and $\varphi$.
\end{itemize}
\end{thm}

Given two acyclic carriers $F,G:\Phi\rightrightarrows\Gamma$,
we write $F\subseteq G$ whenever $F(\sigma)\subseteq G(\sigma)$
for all $\sigma \in \Phi$. Given a pair of acyclic
carriers $F:\Phi \rightrightarrows\Gamma$ and $H:\Gamma \rightrightarrows \Psi$,
we also define the composition carrier $H\circ F:\Phi \rightrightarrows \Psi$, where
each $\sigma \in \Phi$ is sent to
$$
H\circ F(\sigma) \coloneqq \bigcup_{\tau \in F(\sigma)} H(\tau)\ .
$$
In particular, notice that if $f$ is carried by $F$ and $g$ is carried by $G$,
then $g\circ f$ is `carried' by $G\circ F$. Note, however, that this
composition does not need to be acyclic!

\begin{exmp}
Consider a regular morphism $f:\Phi \to \Gamma$. We can then define the
(not necessarily acyclic) carrier $F_f:\Phi\rightrightarrows \Gamma$ induced by
$f$ as the assignement sending $\sigma \in \Phi$ to $\overline{f(\sigma)}$. Notice
that by continuity of $f$ we have that for any pair $\tau \prec \sigma$ in $\Phi$,
we have that $\overline{f(\tau)}\subseteq \overline{f(\sigma)}$. Also,
$\overline{f(\sigma)} \neq \emptyset$ since it must contain
at least a point. Given an acyclic carrier $G:\Gamma \rightrightarrows \Psi$,
we will denote by $G(f(\sigma))$ the composition of carriers $G\circ F_f(\sigma)$
for all $\sigma \in \Phi$. This situation will come up very often in this text
and whenever we are looking at the composition $G\circ F_f$ we will assume that
it is acyclic. Note that $F_f$ is acyclic if $f$ is an embedding of the regular CW-complex $\Phi$ as a subcomplex of $\Gamma$.
\end{exmp}

\subsection{Regular diagrams of filtered complexes}
\label{sub:reg-diag-complexes}
We will first recall a few gluing constructions
that one can perform in algebraic topology. For a brief introduction to these see~\citep[App.~4.G]{Hatcher2002}.
They are also relevant in Kozlov's approach~\citep{Kozlov2008}, where
diagrams of spaces over trisps are studied.

Let $K$ be a simplicial complex. We may view $K$ as a category in several ways:
The first construction has the $0$-simplices of $K$ as its objects and a unique
morphism $v_0 \to v_1$  if and only if either $v_0 = v_1$ or there is a chain of
$1$-simplices connecting the two. The composition of morphisms is completely
fixed by uniqueness. If we apply this construction to the barycentric subdivision
$\Bd(K)$ of $K$, we end up with another category defined as follows: The objects
are given by the simplices in $K$ and there is a unique morphism $\tau \to \sigma$
if $\tau$ is adjacent to $\sigma$. In the following we will mean the second
construction, when we refer to $K$ as a category, and we will denote the morphism
$\tau \to \sigma$ by $\tau \prec \sigma$.
The following will be our main object of study.
\begin{defn}
  Let $K$ be a simplicial complex. A contravariant functor $\cD:K\to \CWcpx$ is
  called a \emph{regular diagram of CW-complexes} and its
  category is denoted by $\RDiag{K}$.
  A \emph{regularly filtered regular diagram of CW-complexes}~$\cD$ over $K$ is a
  contravariant functor $\cD:K\rightarrow \RCWcpx$, we will denote this category by $\RRDiag{K}$.
  A morphism $f:\cD \rightarrow \cL$ between a
  pair of diagrams over $K$ consists of natural transformations
  $f:\cD(\sigma) \rightarrow \cL(\sigma)$ for all $\sigma \in K$ such that
  $f \circ \cD(\tau \prec \sigma) = \cL(\tau \prec \sigma) \circ f$ for all
  $\tau \prec \sigma$ in $K$. On the other hand if we restrict
  to contravariant functors $\cD:K\rightarrow \FCWcpx$ we will call $\cD$ a
  \emph{filtered regular diagram of CW-complexes} denoting the corresponding
  category by $\FRDiag{K}$. If for a diagram $\cD \in \FRDiag{K}$ the maps
  $\cD(\tau \prec \sigma)$ are inclusions respecting the cellular structures
  for all $\tau \prec \sigma$
  from $K$, then we call $\cD$ a \emph{fully filtered diagram of CW-complexes}.
  We have embeddings of categories
  \[
  \FFDiag{K} \subset \FRDiag{K} \subset \RRDiag{K}
  \]
  for all simplicial complexes $K$.
\end{defn}

\begin{exmp}
\label{ex:cover-spaces}
Consider a filtered CW-complex $X$ covered by filtered subcomplexes $\cU=\big\{U_i
\big\}_{i\in I}$. We define $X^\cU$ over the nerve $N_\cU$ as
$X^\cU(J) = \bigcap_{j \in J} U_j$. This diagram $X^\cU$ is part of $\FFDiag{N_\cU}$
since all morphisms $X^\cU(I \subseteq J)$ are actually
embeddings of subcomplexes. On the other hand, we can define
the constant diagram $*^\cU$ as $*^\cU(J)_r = *$  if
$X^\cU(J)_r \neq \emptyset$ or $*^\cU(J)_r = \emptyset$ otherwise; for all
$J \in N_\cU$ and all $r \in \bR$. We also have that $*^\cU$ is in $\FFDiag{N_\cU}$.
Then, there is an obvious epimorphism of diagrams
$X^\cU \rightarrow *^\cU$. Continuing with the same example, we can also define
the complex of spaces $\pi_0^\cU$ given by ${\pi_0}^\cU(J) = \pi_0(U_J)$ for all
$J \in N_\cU$. Each $\pi_0(U_J)$ is a disjoint union of points
that are identified with each other as the filtration value increases. Notice that in this case $\pi_0^\cU \in \RRDiag{K}$. Altogether we have a
sequence of epimorphisms $X^\cU \rightarrow {\pi_0}^\cU \rightarrow *^\cU$.
\end{exmp}

\subsection{Geometric Realization}
For an abstract simplicial complex $K$, we denote by $|K|$ its underlying topological space.
Given a $k$-simplex $\sigma \in K$, we write $|\sigma|$ to denote the number of
vertices of $\sigma$. We will use $\dim(\sigma)$ for the dimension of a simplex
$\sigma$, that is $\dim(\sigma)=|\sigma|-1$.
We write by $\Delta^n$ the topological space associated to the standard $n$-simplex.
Given a simplex $\sigma \in K$, we will use the notation
$\Delta^\sigma \coloneqq \Delta^{\dim(\sigma)}$ for simplicity.
Given a pair $\tau \prec \sigma$ in $K$, we have a corresponding inclusion
$\Delta^\tau \hookrightarrow \Delta^\sigma$.
As a special case of a CW-complex, we will denote by $K^n$ and $K^{\leq n}$ the
set of $n$-cells and the $n$-skeleton respectively.

\begin{defn}
\label{defn:Geom-Real}
Let $\cD \in \RDiag{K}$.
The \emph{geometric realization} $\Delta_K \cD$ of $\cD$ is the object in
$\CWcpx$ defined as
$$
\Delta_K \cD = \bigsqcup_{\sigma \in K} \Delta^\sigma \times \cD(\sigma) \big/ \sim \ ,
$$
where, for any pair $\tau \preceq \sigma$ in $K$ the relation identifies a pair of points
\[
(\Delta^\tau \hookrightarrow \Delta^\sigma)(x) \times y \sim x \times \cD(\tau \preceq \sigma)(y)
\]
for each pair of points $x \in \Delta^\tau$ and $y \in \cD(\sigma)$.  This  $\Delta_K \cD$ has a natural filtration given by
$F^p \Delta_K \cD = \bigcup_{\sigma \in K^{\leq p}} \Delta^\sigma \times \cD(\sigma)$
for all $p \geq 0$.
A cell $\tau \times c$ is a face of another cell $\sigma \times a$ if and only if $\tau \preceq \sigma$ and also  $c\in \overline{\cD(\tau \preceq \sigma)(a)}$. If the underlying simplicial complex $K$ is clear from the context, we will write $\Delta \cD$ instead of $\Delta_K \cD$.
\end{defn}

Notice that Definition~\ref{defn:Geom-Real} also applies to diagrams $\cD \in \RRDiag{K}$. We define $\Delta_K\cD$ by setting $(\Delta_K\cD)_r \coloneqq \Delta_K (\cD_r)$ for all $r \in \bbR$. Notice that our gluing conditions are consistent in this case as
\[
\cD(\tau \preceq \sigma) \circ \Sigma^{t}\cD(\sigma)(y) = \Sigma^{t}\cD(\tau) \circ \cD(\tau \preceq \sigma)(y)
\]
for any pair $\tau \preceq \sigma$ from $K$ and all $t>0$ and all points $y \in \cD(\sigma)$.
Altogether we obtain $\Delta_K(\cD) \in \RCWcpx$.
Given a regular morphism $\cF:\cD \rightarrow \cL$ of diagrams in $\RRDiag{K}$, there is an induced morphism on the geometric realization which we denote $\Delta\cF$. Denote by $*^\cD$ the diagram given by
\[
    *^\cD(\sigma)_r =
    \begin{cases}
        * & \text{if } \cD(\sigma)_r \neq \emptyset \\
        \emptyset & \text{else}
    \end{cases}
\]
and note that there is a homotopy equivalence $\Delta (*^\cD)_r\simeq \lvert K^\cD_r \rvert$, where $K^\cD$ is the filtered simplicial complex with the same underlying vertex set as $K$ and $\sigma \in K^\cD_r$ if and only if $\cD(\sigma)_r \neq \emptyset$.
The projection onto the simplex coordinates gives a \emph{base projection}
$p_b:\Delta\cD \rightarrow \Delta (*^\cD) \simeq \lvert K^\cD \rvert$.

\begin{exmp} Let $\cD \in \FRDiag{K}$. We define the \emph{multinerve} of $\cD$ as
$$
\MN(\cD) = \Delta(\pi_0(\cD))\ .
$$
This object was first introduced in~\citep{ColVer2014} in the case of $\pi_0^\cU$ for a space $X$ covered by~$\cU$. In~\citep{ColVer2014} it was defined as a
simplicial poset, a notion that is equivalent to that of a $\Delta$-complex.
There are epimorphisms $\Delta \cD \rightarrow \MN(\cD) \rightarrow \Delta(*^\cD)\simeq |K|$.
\end{exmp}

\begin{rem}
Let $\cD$ be a diagram of CW-complexes over the simplicial complex $K$. We can extend $\cD$ to a diagram $\cD'$ on the barycentric subdivision $\Bd(K)$ by defining
\(
    \cD'(\tau_0 \prec \dots \prec \tau_n) = \cD(\tau_n)
\)
on an $n$-simplex $\tau_0 \prec \tau_1 \prec \dots \prec \tau_n$ in $\Bd(K)$. A non-identity morphism in $\Bd(K)$ that has $\tau_0 \prec \tau_1 \prec \dots \prec \tau_n$ as its codomain must have the same flag with one of the $\tau_k$'s left out as its domain. The diagram $\cD'$ maps such a morphism to the identity in case $k \neq n$ or the morphism $\cD(\tau_{n-1} \prec \tau_{n})$ in case $k = n$. It is clear from the definition of the homotopy colimit via the simplicial replacement that the geometric realization $\Delta(\cD')$ is homotopy equivalent to $\hocolim\cD$. A modified version of the homotopy equivalence $\lvert K \rvert \simeq \lvert \Bd(K) \rvert$ shows that $\Delta(\cD) \simeq \Delta(\cD')$. Hence, we could have worked with homotopy colimits all throughout, but we chose to work with the geometric realization since it is technically easier to handle and because in some instances it is the Mayer-Vietoris blowup complex, which has already appeared before in Topological Data Analysis~\citep{ZomCar2008}.
\end{rem}

\begin{prop}
\label{prop:compatible-homotopies}
Let $\cF:\cD \rightarrow \cL$ be a morphism of diagrams in $\RDiag{K}$. If
$\cF(\sigma)$ is a homotopy equivalence for all $\sigma \in K$, then
$\Delta\cF:\Delta\cD \rightarrow \Delta \cL$ is a homotopy equivalence.
\end{prop}

One way to see this is to view $\Delta \cD$ as a homotopy colimit (see last remark), which is a homotopy invariant functor on diagrams. Also, a proof of this result in the more general context of diagrams of spaces can be found in~\citep[Prop.~4G.1]{Hatcher2002}.

\begin{exmp}
Let $X \in \CWcpx$ covered by $\cU$ and recall the diagram $X^\cU$ from Example~\ref{ex:cover-spaces}. In this case $\Delta(X^\cU)$ is the Mayer-Vietoris blowup complex associated to the pair $(X,\cU)$ and it can be described as a subspace of the product $X \times |N_\cU|$. This leads to the \emph{fiber projection} $p_f:\Delta(X^\cU)\to X$ and to the \emph{base projection} $p_b:\Delta(X^\cU)\to |N_\cU|$.
As shown in~\citep[Prop.~4G.2]{Hatcher2002}, $p_f$ is a homotopy equivalence $\Delta(X^\cU)\simeq X$. If each $X^\cU(\sigma)$ is contractible for all $\sigma \in N_\cU$, then $p_b$ is also a homotopy equivalence by Proposition~\ref{prop:compatible-homotopies}.
\end{exmp}

Proposition~\ref{prop:compatible-homotopies} is of fundamental importance in applied topology, for example in the persistence nerve lemma from~\citep{ChazOudo2008}. An interesting direction of research would be to use this result to define compatible \emph{collapses}, such as in Discrete Morse Theory~(see \citep{Nanda2012} and~\citep{Bauer2011}) and end up with a diagram of regular CW-complexes. This motivates the study of spectral sequences associated to such diagrams. We will see further reasons in Section~\ref{sec:spectral_sequences}.
On the other hand,
given the importance of Proposition~\ref{prop:compatible-homotopies}, we would like to adapt it to an approximate version in the context of diagrams in $\RRDiag{K}$. Instead of studying
homotopy equivalences, we will consider equivalences induced by acyclic carriers. This will be done in Section~\ref{sec:intBlowCpx}.

\subsection{Spectral Sequences of Bounded Filtrations}
Let $A_*$ be a graded module with differentials $d_n:A_n \to A_{n-1}$ for all $n\geq 1$, and such that $A_m=0$ for all $m<0$. Assume that there is a filtration
$0 = F^{-1} A_* \subseteq F^0 A_* \subseteq F^1 A_* \subseteq \cdots \subseteq F^N A_* = A_*$ of $A_*$ that is preserved by the differentials $d_*$ in the sense that $d_n (F^p A) \subseteq F^p A$ for all $p \geq 0$. We say that $A_*$ is a \emph{filtered differential graded module} and denote this by the triple $(A, d, F)$. Then there is a spectral sequence
\[
E^1_{p,q} = \Ho_q\big( F^p A_* \big/ F^{p-1}A_*\big) \Rightarrow \Ho_{p+q}\big( A_* \big)
\]
for all $p,q \geq 0$, see~\citep[Thm. 2.6]{McCleary2001}.
A morphism of spectral sequences is a sequence of bigraded morphisms $f^r :E^r_{p,q} \to \overline{E}^r_{p,q}$ that commute with the spectral sequence differentials, ie.\ $d_r \circ f^r = d_r \circ f^r$ for all $r \geq 0$. Apart from that, these morphisms satisfy
$f^{r+1} = H(f^r)$ for all $r \geq 0$.

Suppose that $(\overline{A}_*, \overline{d}, \overline{F})$ is another filtered differential graded module together with its corresponding spectral sequence $\overline{E}^r_{p,q}$. Consider a morphism $f:A_* \to B_*$ that commutes
with the differential $f\circ d = \overline{d} \circ f$ and also preserves filtrations
$f(F^p A_*)\subseteq \overline{F}^p(\overline{A}_*)$ for all $p\geq 0$.
This induces a morphism of spectral sequences
\[
E^r_{p,q}\to \overline{E}^r_{p,q}
\]
by~\citep[Thm. 3.5]{McCleary2001}. We will denote by \SpSq~the category of spectral
sequences, while we will denote by \PSpSq~the category of functors $F:\bbR \rightarrow \SpSq$.

\section{Spectral Sequences for Geometric Realizations}
\label{sec:spectral_sequences}

Recall the persistent Mayer-Vietoris spectral sequence~\citep{TorrasCasas2019}
associated to a pair $(X, \cU)$ of a space with a cover:
\begin{equation}
\label{ex:traditional-MV-ss}
    E^1_{p,q}(X,\cU) = \bigoplus_{\sigma \in N_{\cU}^p} \PH_q(X^{\cU}(\sigma))
    \Rightarrow \PH_{p+q}(\Delta X^\cU)\simeq \PH_{p+q}(X)\ .
\end{equation}
For the details about this spectral sequence in the persistent case we refer the
reader to \citep{TorrasCasas2019}. There are some limitations to the
applicability of this spectral sequence to Vietoris-Rips complexes that were
already pointed out in~\citep{YooGhr2020}: If we choose a cover of a point
cloud $\bX$ and then deduce a cover $\cU$ of the associated Vietoris-Rips
complex $\VR_*(\bX)$ by subcomplexes, then we will only be able to recover
$\PH_{k}(\VR(\bX))$ from $\PH_{k}(\Delta \VR_*(\bX)^\cU)$ for filtration
parameters below an upper bound $R$ determined by the overlaps of the
covering sets. In this section we will present an alternative regular diagram
of CW-complexes that avoids this upper limit problem completely, see
Example~\ref{ex:VR-geom-real}.
%In particular, we will show that for any simplicial complex $K$, there exists a diagram whose geometric realization is homeomorphic to $|K|$.

Before we solve our problem, we need to introduce some chain complexes. We will come
back to the case of filtrations later, but for now we will focus on regular diagrams instead.
Given a diagram $\cD$ in $\RDiag{K}$, we will denote by
$\cD(\tau \preceq \sigma)_*$ the induced morphism of cellular chain complexes
$C^\cell_*(\cD(\sigma))\to C^\cell_*(\cD(\tau))$.
The cellular chain complex $C^\cell_*(\Delta \cD, \delta^\Delta)$ associated to
$\Delta \cD$ is defined as follows: For all $m \geq 0$ we have that
$C^\cell_m(\Delta \cD)$ is a vector space generated by cells $\sigma \times c$
with $\dim(\sigma) = p$ and $c \in \cD(\sigma)_q$ so that $p+q=m$.
On such a cell $\sigma \times c$ the differential $\delta^\Delta$ is given by
\[
\delta^\Delta(\sigma \times c)= \sum_{\sigma_i \prec \sigma} (-1)^i \left(
\sum_{a \in \overline{\cD(\sigma_i \preceq \sigma)(c)}} [a:\cD(\sigma_i \preceq
\sigma)(c)] \sigma_i \times a \right) +
(-1)^{\dim(\sigma)}\sum_{b \in \overline{c}\setminus c} [b : c] \sigma \times b
\]
where $d_q$ is the differential
$d_q:C_q(\cD(\sigma))\rightarrow C_{q-1}(\cD(\sigma))$ and the sum runs over the
faces $\sigma_i$ of $\sigma$. As we will see in the proof of Lemma~\ref{lem:iso-geom-total}
the map $\delta^\Delta$ is indeed a differential.
In addition, notice that the filtration of $\Delta_K(\cD)$ carries over to
$C_*^\cell{}(\Delta_K\cD)$ by taking
$F^p C_*(\Delta_K\cD) \coloneqq C_*(F^p\Delta_K\cD)$ for all $p\geq 0$.

Now, consider the double complex $(C_{p,q}(\cD), d^V, d^H)$ given by
\[
C_{p,q}(\cD) = \bigoplus_{\sigma \in K^p} C^\cell_q\big(\cD(\sigma)\big)
\]
for all $p,q \geq 0$. The vertical differential is defined by the direct sum of
chain differentials $d^V_{p,q}=(-1)^p\bigoplus_{\sigma \in K^p}d_q^{\sigma}$
where $d^\sigma_*$ denotes the differential from $C^\cell_*(\cD(\sigma))$
for all $\sigma \in K^p$.
The horizontal differential is given by the \cech~differential $d^H_{p,q}$
which is defined for a cell $a \in \cD(\sigma)$ as
$\sum_{\sigma_i \prec \sigma}(-1)^i \cD(\sigma_i \prec \sigma)^*(a)$,
where $\cD(\sigma_i \prec \sigma)^*$ denotes the induced chain
morphism $C^\cell_*(\cD(\sigma))\to C^\cell_*(\cD(\sigma_i))$ for all
faces $\sigma_i$ from $\sigma$. Of course
$d^V\circ d^V = 0$ and $d^H \circ d^H = 0$ by functoriality of
$C^\cell_*(\cdot)$ and the fact that
$\cD(\rho \prec \tau)\cD(\tau \prec \sigma)=\cD(\rho \prec \sigma)$ for any
three simplices $\rho \prec \tau \prec \sigma$. On the other hand,
anticommutativity $d^V \circ d^H = - d^H \circ d^V$ follows since
$\cD(\tau \prec \sigma)^*$ is a chain morphism for all $\tau \prec \sigma$
from $K$.

Now, we consider the double complex spectral sequence from~\citep[Section 2.4]{McCleary2001}.
Given $\cD$ in $\RDiag{K}$ there is a spectral sequence
\[
E^1_{p,q}(\cD) = \bigoplus_{\sigma \in K^p} \Ho_q(\cD(\sigma)) \Rightarrow \Ho_{p+q}(S^\Tot_*(\cD))
\]
where $S^\Tot(\cD)$ is the \emph{total complex} defined as
$S^\Tot_n(\cD) = \bigoplus_{p+q=n} C_{p,q}(\cD)$ together with a differential
$d^\Tot = d^V + d^H$.
Also, recall that the total complex has a filtration induced by the vertical
filtration on $C_{p,q}(\cD)$ given by
\[
F^m S^\Tot_*(\cD) = \bigoplus_{\substack{p+q=n\\ p \leq m}} C_{p,q}(\cD)
\]
for all integers $m \geq 0$, see~\citep{TorrasCasas2019} for an explanation.
We will now relate this total complex to the geometric realization from
Definition~\ref{defn:Geom-Real}.
\begin{lem}
  \label{lem:iso-geom-total}
  There is an isomorphism
$C_*^\cell(\Delta \cD, \delta^\Delta) \simeq S_*^\Tot(\cD)$ which preserves
filtration. That is, $F^pC_*^\cell(\Delta \cD, \delta^\Delta) \simeq
F^pS_*^\Tot(\cD)$ for all $p\geq 0$.
\end{lem}
\begin{proof}
First we define a chain morphism $\psi:C^\cell_m(\Delta \cD)\to S_m^\Tot(\cD)$
generated by the assignment: a cell $\sigma \times c \in (\Delta \cD)_m$ with
$\sigma \in K^p$ and $c \in \cD(\sigma)^q$ for integers $p+q=m$, is sent to
$\psi(\sigma\times c)= (c)_\sigma \in S^\Tot_m(\cD)$. Here we denote by
$(c)_\sigma$ the vector in $S^\Tot_*(\cD)$ that is zero everywhere except at the
entry $C^\cell_q(\cD(\sigma))$. On the other hand, $\psi$ is a chain morphism since
we have the equality
\begin{multline}
\nonumber
\psi\big(\delta^\Delta(\sigma \times c)\big) =
\sum_{\sigma_i \prec \sigma} (-1)^i \Big(
\sum \limits_{a \in \overline{\cD(\sigma_i \preceq \sigma)(c)}} ([a:\cD(\sigma_i \preceq \sigma)(c)] a)_{\sigma_i} \Big) \\ + (-1)^{\dim(\sigma)}\sum_{b \in \overline{c}\setminus c} ([b : c]  b)_\sigma =
\sum_{\sigma_i \prec \sigma} (-1)^i (\cD(\sigma_i \preceq \sigma)^*(c))_{\sigma_i} +
(-1)^{\dim(\sigma)}(d^\sigma(c))_\sigma \\
= (d^H + d^V)((c)_\sigma) = d^\Tot((c)_\sigma).
\end{multline}
One can see that $\psi$ is injective, and admits an inverse
$\psi^{-1}:S_m^\Tot(\cD) \to C^\cell_m(\Delta \cD)$ that sends $(\sigma)_c$ to
$\sigma \times c$. Notice that by definition $\psi$ sends a chain in
$F^p C_n^\cell(\Delta \cD)$ to a chain in $F^pS_n^\Tot(\cD)$ for all $p\geq 0$
and in particular it preserves filtration.
\end{proof}

\begin{rem}
Continuing with the remark at the end of
subsection~\ref{sub:reg-diag-complexes}, we could have considered the homotopy
colimit spectral sequence
$$
E^1_{p,q}(\Bd(K), \cD') = \bigoplus_{\sigma \in \Bd(K)^p}\Ho_q(\cD'(\sigma))
\Rightarrow \Ho_{p+q}(\hocolim \cD).
$$
\end{rem}

Let us construct a diagram of spaces whose geometric realization is homeomorphic
to $|K|$ for any finite simplicial complex $K$.
We start by taking a finite partition $\cP$ of the vertex set $V(K)$ and denote
by $K(U)$ the maximal subcomplex of $K$ with vertices in $U \in \cP$.
We will denote by $\Delta^\cP$ the standard simplex with vertices in~$\cP$.
For a simplex $\tau \in K$, we define $\cP(\tau) \in \Delta^\cP$ to be the
simplex consisting of all partitioning sets $U \in \cP(\tau)$ such that
$\tau \cap U\neq \emptyset$.
%such that for all partitioning sets $U \in \cP(\tau)$ one has .
In particular if $U \in \cP(\tau)$, then it determines a standard simplex
$\tau(U) \in K(U)$ of dimension $|\tau \cap U|-1\geq 0$. For a vertex $v \in K$,
we will denote by $\cP(v)$ the partitioning set from $\cP$ which
contains $v$.

We define the $(K,\cP)$-\emph{join} diagram $\cJ^K_\cP:\Delta^\cP \rightarrow
\FCWcpx$ for all $\sigma \subseteq \cP$ by assigning
the subspace
\[
\cJ^K_P(\sigma) = \bigcup_{\substack{\rho \in K\\ \cP(\rho)=\sigma}}\prod_{U \in \sigma}\Delta^{\rho(U)}
\hookrightarrow \prod_{U\in \sigma} |K(U)|\ ,
\]
for all $\sigma \in \Delta^\cP$.
For any pair $\tau \preceq \sigma$ in $\Delta^\cP$, the projection
$\pi_{\tau\preceq \sigma}:\prod_{U \in \sigma}|K(U)| \rightarrow \prod_{U \in \tau}|K(U)|$
induces a regular morphism $\cJ^K_\cP(\tau \preceq \sigma):\cJ^K_\cP(\sigma)\rightarrow \cJ^K_\cP(\tau)$.

\begin{lem}
\label{lem:iso-join-diagram}
Let $K$ be a simplicial complex together with a partition $\cP$ of its vertex set $V(K)$. There is a CW-complex homeomorphism $\Delta(\cJ^K_\cP) \simeq |K|$.
\end{lem}
\begin{proof}
Consider the continuous map $f:\Delta(\cJ^K_\cP)\to |K|$ defined by mapping a point
\[
\Big(\sum_{U\in \cP(\tau)} y_U U,
\Big( \sum_{v \in U}x_v v\Big)_{U \in \cP(\tau)}  \Big) \in \Delta^{\cP(\tau)}\times
\prod_{U \in \cP(\tau)}\Delta^{\tau(U)} \bigg/\sim
\]
to $\sum_{v \in \tau} y_{\cP(v)} x_v v$ in $\Delta^\tau$ for
all $\tau \in K$; where we have
values $0 \leq y_U \leq 1$ and $0 \leq x_v \leq 1$ for all $U \in \cP(\tau)$ and all $v \in U$, and such that
$\sum_{U\in \cP(\tau)} y_U = 1$ and $\sum_{v \in U}x_v = 1$ for all $U \in \cP$.
On the other hand, let $\sum_{v \in \tau}x_v v \in \Delta^\tau$ be a point such
that $0 \leq x_v \leq 1$ for all $v \in \Delta^\tau$
and such that $\sum_{v \in \tau}x_v = 1$. Then we
can define the inverse continuous map
\[
f^{-1}\Big(\sum_{v \in \tau}x_v v \Big) =
\bigg( \sum_{U\in \cP(\tau)} \Big( \sum_{v \in U}x_v\Big) U,
\bigg( \psi_U \Big(\sum_{v \in \tau}x_v v \Big) \bigg)_{U\in \cP(\tau)}\bigg)
\]
where
\[
\psi_U \Big(\sum_{v \in \tau}x_v v \Big) =
\begin{cases}
  \sum_{v \in \tau}\bigg(\dfrac{x_v}{\sum_{v \in \tau(U)}x_v}\bigg)v
  \mbox{ if } \sum_{v \in \tau(U)}x_v \neq 0\\
  *  \in \Delta^{\tau(U)} \mbox{ otherwise (see below),}
\end{cases}\ .
\]
If $\sum_{v \in \tau(U)}x_v = 0$, then $x_v = 0$ for all $v \in \tau(U)$ and the $U$-coordinate of the simplex $\tau$ in $\Delta^{\cP(\tau)}$ is $0$, which means that $\tau(U)$ is collapsed to a single point by the equivalence relation used to define $\Delta(\cJ^K_\cP)$.
%where we have used $*  \in \Delta^{\tau(U)}$ to denote the fact that
%$\Delta^{\tau(U)}$ gets collapsed to a single point.
It is straightforward to check that $f$ and $f^{-1}$ are well defined and consistent
along $K$.
\end{proof}

\begin{exmp}
  Consider a simplicial complex $K$ depicted in the top left part of Figure~\ref{fig:KPjoin-illustrated}.
  We consider a partition of the vertex set of $K$ into the two subsets $\cP=\{U,V\}$,
  where points in $U$ are indicated by black circles and points in $V$ are
  indicated by red squares. In the top right of Figure~\ref{fig:KPjoin-illustrated},
  we depict the standard $1$-simplex $\Delta^\cP$ together with the diagram
  $J^K_\cP$ over it. In particular, notice that $J^K_\cP([U,V])$ is a subset
  of the product $|K(U)| \times |K(V)|$ and that the
  morphisms $J^K_\cP([U,V])\to J^K_\cP(V) = |K(V)|$ and $J^K_\cP([U,V])\to J^K_\cP(U) = |K(U)|$
  are both projections. Finally, the bottom left of Figure~\ref{fig:KPjoin-illustrated}
  shows the geometric realization $\Delta J^K_\cP$, where each green line and
  each red dashed line gets ``squashed'' to a single point.

  \begin{figure}[ht]
  \begin{center}
  \begin{tikzpicture}
  \node at (-1.5,-1) {\Large $K$};
  \node at (1,-1) {
    \begin{tikzpicture}
      % interiors
      \fill[blue!20] (0,0)--(2,1)--(0,1)--cycle;
      \fill[blue!40] (0,1)--(2,1)--(2,2)--(0,2)--cycle;
      % vertices
      \node[fill=black,circle,inner sep=3pt] (A) at (0,0) {};
      \node[fill=black,circle,inner sep=3pt] (B) at (0,1) {};
      \node[fill=black,circle,inner sep=3pt] (C) at (0,2) {};
      \node[draw=black,line width=0.5mm,fill=red,rectangle,inner sep=4pt] (D) at (2,1) {};
      \node[draw=black,line width=0.5mm,fill=red,rectangle,inner sep=4pt] (E) at (2,2) {};
      % edges
      \draw[line width=0.5mm] (A) -- (C)-- (E)--(D)--(A);
      \draw[line width=0.5mm] (D)--(B)--(E);
      \draw[line width=0.5mm, dashed] (C)--(D);
      % partition nodes
      \node at (0,2.7) {\bf $U$};
      \node at (2,2.7) {\bf $V$};
    \end{tikzpicture}
  };
  \node at (4.2,-2.2) {$J^K_\cP(U)$};
  \node[rotate=45, fill=green!20] (JU) at (4.5,-3.5) {
    \begin{tikzpicture}
      \draw[line width=0.5mm] (0,0)--(0,2);
      \fill (0,0) circle (3pt);
      \fill (0,1) circle (3pt);
      \fill (0,2) circle (3pt);
    \end{tikzpicture}
  };
  \node at (8.8,-2.2) {$J^K_\cP(V)$};
  \node[rotate=45, fill=green!20] (JV) at (8.5,-3.5) {
    \begin{tikzpicture}
      \draw[line width=0.5mm] (0,0)--(1,0);
      \node[draw=black,line width=0.5mm,fill=red,rectangle,inner sep=4pt] at (0,0) {};
      \node[draw=black,line width=0.5mm,fill=red,rectangle,inner sep=4pt] at (1,0) {};
    \end{tikzpicture}
  };
  \node at (6.5,0.2) {$J^K_\cP([U,V])$};
  \node[rotate=45, fill=gray!20] (JUV) at (6.5,-1.5) {
    \begin{tikzpicture}
      % draw interiors
      \fill[blue!20] (0,1)--(0,2)--(1,2)--(1,1)--cycle;
      \fill[gray!20] (0,0)--(0,1)--(1,1)--(1,0)--cycle;
      % draw edges
      \draw[line width=0.5mm] (0,1)--(0,2)--(1,2)--(1,1)--cycle;
      \draw[line width=0.5mm] (0,0)--(0,1);
      \draw[line width=0.5mm, dashed, gray] (0,0)--(1,0)--(1,1);
      % draw nodes
      \fill (0,0) circle (3pt);
      \fill (0,1) circle (3pt);
      \fill (0,2) circle (3pt);
      \filldraw[draw=gray, line width=0.5mm,fill=gray!30] (1,0) circle (3pt);
      \fill (1,1) circle (3pt);
      \fill (1,2) circle (3pt);
    \end{tikzpicture}
  };
  \draw[line width=0.5mm] (4.5,-5.5) -- (8.5,-5.5);
  \node (NUV) at (6.5,-5.5) {};
  \node[anchor=north] at (6.5,-5.8) {\bf $[U,V]$};
  \node[fill=black, circle, inner sep=3pt] (NU) at (4.5,-5.5) {};
  \node[anchor=north] at (4.5,-5.8) {\bf $U$};
  \node[draw=black,line width=0.5mm,fill=red, rectangle, inner sep=4pt] (NV) at (8.5,-5.5) {};
  \node[anchor=north] at (8.5,-5.8) {\bf $V$};
  \node at (10,-5) {\Large $\Delta^\cP$};
  \node at (10,-1) {\Large $J^K_\cP$};
  % draw arrows and lines
  \draw[->, line width=0.5mm] (JUV)--(JU);
  \draw[->, line width=0.5mm] (JUV)--(JV);
  \draw[line width=1mm, green!20] (NU)--(JU);
  \draw[line width=1mm, green!20] (NV)--(JV);
  \draw[line width=1mm, gray!20] (NUV)--(JUV);
  \node at (-1.5,-5) {\Large $\Delta J^K_\cP$};
  \node (GR) at (1,-5) {
    \begin{tikzpicture}
      % fill
      \fill[blue!20] (0,0)--(1,0.5)--(2.5,0.5)--(1.5,0)--cycle;
      \fill[blue!40] (1,0.5)--(2.5,0.5)--(2.5,2)--(1,2)--cycle;
      \fill[blue!20] (1,2)--(2,2.5)--(3.5,2.5)--(2.5,2)--cycle;
      \fill[blue!10] (3.5,2.5)--(2.5,2)--(2.5,0.5)--(3.5,1)--cycle;
      % horizontal identifications
      \draw[green!80!black, line width=0.5mm] (1.5,0)--(3.5,1);
      \draw[green!80!black, line width=0.5mm] (2.5,2)--(3.5,2.5);
      \draw[green!80!black, line width=0.5mm] (2.5,1.7)--(3.5,2.2);
      \draw[green!80!black, line width=0.5mm] (2.5,1.4)--(3.5,1.9);
      \draw[green!80!black, line width=0.5mm] (2.5,1.1)--(3.5,1.6);
      \draw[green!80!black, line width=0.5mm] (2.5,0.8)--(3.5,1.3);
      % vertical identifications
      \draw[red!80!black, line width=0.5mm, dashed] (1,0.5)--(1,2);
      \draw[red!80!black, line width=0.5mm, dashed] (1.2,0.6)--(1.2,2.1);
      \draw[red!80!black, line width=0.5mm, dashed] (1.4,0.7)--(1.4,2.2);
      \draw[red!80!black, line width=0.5mm, dashed] (1.6,0.8)--(1.6,2.3);
      \draw[red!80!black, line width=0.5mm, dashed] (1.8,0.9)--(1.8,2.4);
      \draw[red!80!black, line width=0.5mm, dashed] (2,1)--(2,2.5);
      % draw nodes
      \fill (0,0) circle (2pt);
      \fill (1,0.5) circle (2pt);
      \fill (2,1) circle (2pt);
      \fill (1.5,0) circle (2pt);
      \fill (2.5,0.5) circle (2pt);
      \fill (3.5,1) circle (2pt);
      \fill (1,2) circle (2pt);
      \fill (2,2.5) circle (2pt);
      \fill (2.5,2) circle (2pt);
      \fill (3.5,2.5) circle (2pt);
    \end{tikzpicture}
  };
  \end{tikzpicture}
  \end{center}
  \caption{Illustration of an example of a $J^K_\cP$ diagram.}
  \label{fig:KPjoin-illustrated}
  \end{figure}
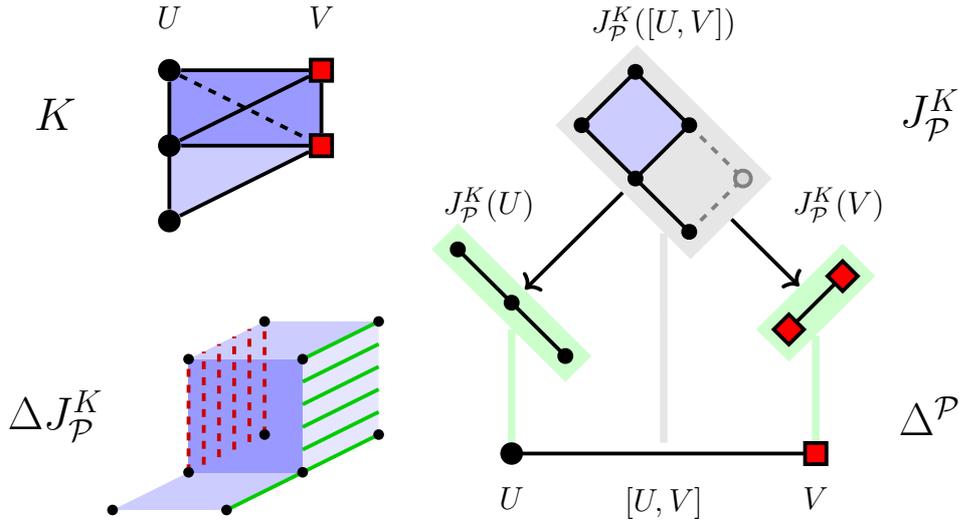
\end{exmp}

Note that $\cJ^K_\cP$ is not a diagram of simplicial complexes, but of
\emph{prodsimplicial} complexes as in \citep[Def.~2.43]{Kozlov2008}. See Figure~\ref{fig:seven-simplex} for an example of
$\cJ^K_\cP$ where $\cP$ partitions the vertex set of a $7$-simplex.
In particular, one can consider the spectral sequence
\[
E^1_{p,q} (\cJ^K_\cP)=\bigoplus_{\sigma \in \Delta^\cP}
\Ho_q(\cJ^K_\cP(\sigma))\Rightarrow \Ho_{p+q}(K)\ .
\]

\begin{figure}[ht]
\begin{center}
\begin{tikzpicture}
  \node (simplex7) at (-5,0) {
  \begin{tikzpicture}[scale=0.3]
  \foreach \a in {0, 51.43, 102.86, 154.29,
       205.71, 257.14, 308.57}
  {
    \foreach \b in {0, 51.43, 102.86, 154.29,
         205.71, 257.14, 308.57}
    {
      \draw (\a:5)--(\b:5);
    }
  }
  % color covers
  % red triangle
  \filldraw[thick, draw=red, fill=red,semitransparent] (0:5)--(51.43:5)-- (102.86:5)--cycle;
  \fill[red] (0:5) circle (10pt);
  \fill[red] (51.43:5) circle (10pt);
  \fill[red] (102.86:5) circle (10pt);
  % blue edge
  \fill[blue] (154.29:5) circle (10pt);
  \fill[blue] (205.71:5) circle (10pt);
  \draw[blue, thick] (154.29:5)--(205.71:5);
  % green edge
  \fill[green] (257.14:5) circle (10pt);
  \fill[green] (308.57:5) circle (10pt);
  \draw[green, thick] (257.14:5)--(308.57:5);
  \end{tikzpicture}};
\node (ABC) at (0,0) {
\begin{tikzpicture}[scale=0.3]
% vertices
\fill (0,0) circle (5pt);
\fill (1,-1) circle (5pt);
\fill (1.5,1.5) circle (5pt);
\fill (4,0) circle (5pt);
\fill (5,-1) circle (5pt);
\fill (5.5,1.5) circle (5pt);
\fill (0,4) circle (5pt);
\fill (1,3) circle (5pt);
\fill (1.5,5.5) circle (5pt);
\fill (4,4) circle (5pt);
\fill (5,3) circle (5pt);
\fill (5.5,5.5) circle (5pt);
% draw edges
\draw (0,0)--(1,-1)--(1.5,1.5)--cycle;
\draw (4,0)--(5,-1)--(5.5,1.5)--cycle;
\draw (0,4)--(1,3)--(1.5,5.5)--cycle;
\draw (4,4)--(5,3)--(5.5,5.5)--cycle;
\draw (0,0)--(4,0)--(4,4)--(0,4)--cycle;
\draw (1,-1)--(5,-1)--(5,3)--(1,3)--cycle;
\draw (1.5,1.5)--(5.5,1.5)--(5.5,5.5)--(1.5,5.5)--cycle;
\end{tikzpicture}};
\node (BC) at (3,2) {
\begin{tikzpicture}[scale=0.3]
% vertices
\fill (0,0) circle (5pt);
\fill (1,-1) circle (5pt);
\fill (1.5,1.5) circle (5pt);
\fill (4,0) circle (5pt);
\fill (5,-1) circle (5pt);
\fill (5.5,1.5) circle (5pt);
% draw edges
\draw (0,0)--(1,-1)--(1.5,1.5)--cycle;
\draw (4,0)--(5,-1)--(5.5,1.5)--cycle;
\draw (0,0)--(4,0);
\draw (1,-1)--(5,-1);
\draw (1.5,1.5)--(5.5,1.5);
\end{tikzpicture}};
\node (AC) at (3,0) {
\begin{tikzpicture}[scale=0.3]
% vertices
\fill (0,0) circle (5pt);
\fill (4,0) circle (5pt);
\fill (0,4) circle (5pt);
\fill (4,4) circle (5pt);
% draw edges
\draw (0,0)--(4,0)--(4,4)--(0,4)--cycle;
\end{tikzpicture}};
\node (AB) at (3,-2) {
\begin{tikzpicture}[scale=0.3]
% vertices
\fill (0,0) circle (5pt);
\fill (1,-1) circle (5pt);
\fill (1.5,1.5) circle (5pt);
\fill (0,4) circle (5pt);
\fill (1,3) circle (5pt);
\fill (1.5,5.5) circle (5pt);
% draw edges
\draw (0,0)--(1,-1)--(1.5,1.5)--cycle;
\draw (0,4)--(1,3)--(1.5,5.5)--cycle;
\draw (0,0)--(0,4);
\draw (1,-1)--(1,3);
\draw (1.5,1.5)--(1.5,5.5);
\end{tikzpicture}};
\node (A) at (6,2){
\begin{tikzpicture}[scale=0.4]
\draw[thick, green] (0,0)--(4,0);
\fill[green] (0,0) circle (5pt);
\fill[green] (4,0) circle (5pt);
\draw[color=white] (0,-1)--(4,-1)--(4,1)--(0,1)--cycle;
\end{tikzpicture}};
\node (B) at (6,0) {
\begin{tikzpicture}[scale=0.4]
% draw edges
\filldraw[thick, draw=red, fill=red!20] (0,0)--(1,-1)--(1.5,1.5)--cycle;
% vertices
\fill[red] (0,0) circle (5pt);
\fill[red] (1,-1) circle (5pt);
\fill[red] (1.5,1.5) circle (5pt);
\end{tikzpicture}};
\node (C) at (6,-2) {
\begin{tikzpicture}[scale=0.4]
% draw edges
\draw[thick, blue] (0,0)--(0,4);
\draw[color=white] (-1,0)--(-1,5);
% vertices
\fill[blue] (0,0) circle (5pt);
\fill[blue] (0,4) circle (5pt);
\end{tikzpicture}};
% draw arrows
\draw[->, thick] (ABC)--(BC);
\draw[->, thick] (ABC)--(AC);
\draw[->, thick] (ABC)--(AB);
\draw[->, thick] (BC)--(A);
\draw[->, thick] (BC)--(B);
\draw[->, thick] (AC)--(A);
\draw[->, thick] (AC)--(C);
\draw[->, thick] (AB)--(B);
\draw[->, thick] (AB)--(C);
\end{tikzpicture}
\end{center}
\caption{On the left the $7$-simplex together with a partition $\cP$ of its vertex set. On the right, the associated diagram of spaces $\cJ^K_\cP$.}
\label{fig:seven-simplex}
\end{figure}
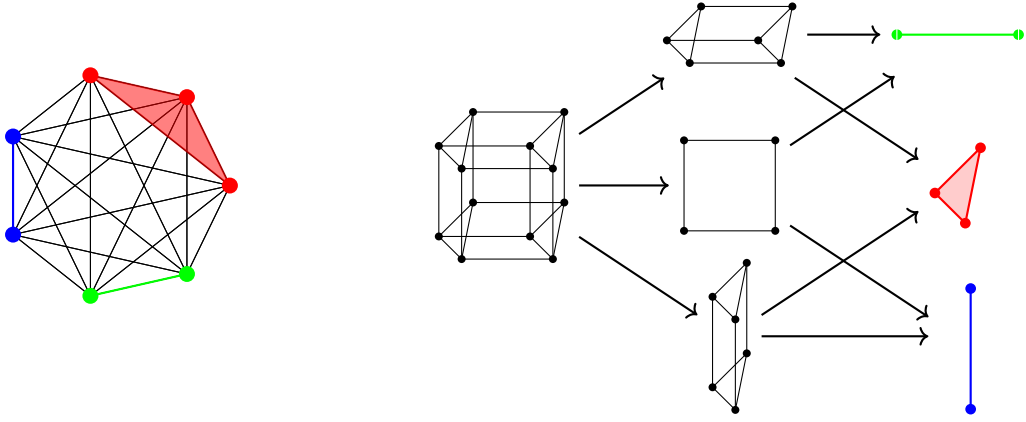

Now, let us consider a filtered simplicial complex $K_* \in \FCWcpx$ such that
its vertex set $V(K_*)$ is fixed throughout all values of $\bbR$. Let $\cP$ be a
partition of $V(K_*)$. We define the filtered regular diagram
$\cJ^K_{\cP} \in \FRDiag{\cP}$ by sending $r \in \bbR$ to $\cJ^{K_r}_{\cP}$. These diagrams
inherit the shift morphisms $\Sigma K_*$ from $K_*$ in the following way: Let $\sigma \in \Delta^\cP$
and notice that we have restrictions $\Sigma^{s-r}K_{|U}:|K_r(U)| \rightarrow |K_s(U)|$
for all $U \in \sigma$, so that we have induced morphisms
\[
\prod_{U \in \sigma }\Sigma^{s-r}K_{|U}: J^{K_r}_\cP(\sigma) \to J^{K_s}_\cP(\sigma)
\]
for all $\sigma \in \Delta^\cP$. In turn, these induce a shift morphism on
$\Delta J^K_\cP$ which respect filtrations, so that we have a commutative diagram
\[
\begin{tikzcd}
E^*_{p,q}(J^{K_r}_\cP) \ar[Rightarrow, r] \ar[d] &
\Delta J^{K_r}_\cP \ar[r, "\simeq"] \ar[d] &
K_r \ar[d] \\
E^*_{p,q}(J^{K_s}_\cP) \ar[Rightarrow, r] &
\Delta J^{K_s}_\cP \ar[r, "\simeq"] &
K_s
\end{tikzcd}
\]
and thus $\PH_*(\Delta \cJ^K_\cP)\simeq \PH_*(K_*)$.
For each simplex $\sigma \in \Delta^\cP$ one can see
$\cJ^{K}_{\cP}(\sigma)$ as a filtered simplicial complex, so that
\[
E^1_{p,q}(\cJ^{K}_{\cP}) = \bigoplus_{\sigma \in (\Delta^\cP)^p}\PH_q(\cJ^K_\cP(\sigma))\Rightarrow \PH_{p+q}(K) \ .
\]

\begin{exmp}
\label{ex:VR-geom-real}
Consider a point cloud $\bX$, a partition $\cP$ and consider its Vietoris Rips complex $\VR_*(\bX) \in \FCWcpx$.
In this case we have a fixed partition of the vertex set of $\VR_*(\bX)$, which allows us to consider the spectral sequence:
\[
E^1_{p,q}\big(\cJ^{\VR_*(\bX)}_\cP\big)=\bigoplus_{\sigma \in \Delta^\cP} \PH_q(\cJ^{\VR_*(\bX)}_\cP(\sigma))\Rightarrow \PH_{p+q}(\VR_*(\bX))\ .
\]
This is very convenient as it avoids the main difficulties with the Mayer-Vietoris blowup complex associated to a cover.
Namely, one recovers $\PH_*(K)$ completely without any bounds depending on the cover overlaps.
In addition, notice that $\Delta \cJ^{\VR_*(\bX)}_\cP$ does not `blowup' more than $\VR_*(\bX)$.
\end{exmp}

The $(K,\cP)$-join diagram is related to~\citep[Example~4]{Robinson2020}. There the motivation behind the filtrations is given by a consistency radius and a filtration based on the differences between local measurements. The same example appears (without a filtration) as one of the opening examples in~\citep[Appendix~4.G]{Hatcher2002}.

\section{$\veps$-acyclic carriers}
\label{sec:eps-carriers}

The following definition will encode our notion of `noise'.

\begin{defn}
Let $X, Y \in \RCWcpx$. An $\veps$-acyclic
carrier $F_*^\veps:X_*\rightrightarrows Y[\veps]_{*}$  is a family of
acyclic carriers $F_a^\veps:X_a \rightrightarrows Y_{a+\veps}$ for all $a \in \bbR$
such that
$$
Y(a+\veps \leq b+ \veps)F_a^\veps(c) \subseteq
F_{b}^\veps (X(a \leq b)(c))
$$
for all cells $c$ of $X_a$ and $a,b \in \bbR$ with $a\leq b$.
\end{defn}

The proposition below is an adaptation of \citep[Thm.~13.4]{Munk1984} or \citep[Lem.~2.4]{CookeFinney1967} to the
context of tame filtered CW-complexes.

\begin{prop}
\label{prop:eps-acyclic-carriers}
Let $X_*, Y_* \in \FCWcpx$ be tame, and assume that there exists an $\veps$-acyclic carrier
\[
    F_*^\veps:X_*\rightrightarrows Y[\veps]_{*}\ .
\]
Then there exist chain morphisms
$f_a^\veps:C_*(X_a) \rightarrow C_*(Y_{a+\veps})$ carried by $F_a^\veps$ for all
$a \in \bbR$, so that $Y(a+\veps \leq b+\veps) \circ f_a^\veps = f_{b}^\veps \circ X(a\leq b)$.
Furthermore, given another such sequence of morphisms $g_a^\veps:C_*(X_a) \rightarrow C_*(Y_{a+\veps})$,
there exist chain homotopy equivalences $H_a^\veps:g_a^\veps \simeq f_a^\veps$ which
are carried by $F_a^\veps$ for all $a \in \bbR$.
\end{prop}

\begin{proof}
Let $b \in \bbR$ and assume that $f^\veps_a$ has already been defined for all
values $a < b$, where we allow for $b = -\infty$.
We first define $f^\veps_b$ on all cells which are in the image of
$X(a < b)$ for any $a < b$ using the definition
$$
    f^\veps_b \circ X(a < b) = Y(a+\veps < b + \veps) \circ f^\veps_{a}\ .
$$
Notice that the assumption that $X_a \subseteq X_b$ is crucial for this to work.
By hypotheses, given a cell $c \in \img(X(a < b))$, its image $f_b^\veps(c)$ is
then contained in
$$
Y(a+\veps < b+ \veps)F_a^\veps(\widetilde{c}) \subseteq
F_{b}^\veps (X(a < b)(\widetilde{c}))\ ,
$$
where $\widetilde{c} \in X_a$ is such that
$c = X(a < b)(\widetilde{c})$.
Hence, $f^\veps_b$ satisfies the carrier condition.
Next we define $f^\veps_b$ on the remaining cells in
$$
\widetilde{X_b} = X_b \setminus \big( \bigcup_{a < b} X(a < b)\big)\ .
$$
We proceed to prove this by induction.
First, choose a $0$-cell $f^\veps_b(v) \in F_b^\veps(v)$ for each remaining $0$-cell $v \in \widetilde{X_b}$,
 and notice that $d_* f^\veps_b(v) = 0 = f^\veps_b(d_* v)$, where we use $d_*$ for the chain complex differentials.
Next, by induction, assume that for a fixed $p\geq 0$, the $p$-cells $s \in X_b$ have image
$f_b^\veps(s)$ carried by $F^\veps_b(s)$ and such that
$d_* \circ f^\veps_b(s) = f^\veps_b \circ d_*(s)$. We would like to extend
$f_b^\veps$ to the $(p+1)$-cells. By semicountinuity, given such a cell $c \in X_b$, its
boundary $d_* c$ will be contained in~$F^\veps_b(c)$. On the other hand, notice that
by linearity and the induction hypotheses
$d_* f^\veps_b(d_* c) = f^\veps_b(d_* d_* c)= 0$, thus
$f^\veps_b(d_* c)$ is a cycle in $F^\veps_b(c)$.
By acyclicity, there exists $h \in F^\veps_b(c)$
such that $d_* h = f^\veps_b(d_* c)$ and thus we can define
$f^\veps_b(c)= h$. Altogether, we have defined a chain morphism $f^\veps_b$
which is carried by $F^\veps_b$. Since $X$ is tame, the statement holds for all
filtration values on $\bbR$.

Now, assume that $g_b^\veps$ is also carried by $F^\veps_b$ for all $b \in \bR$.
Following \citep[Sec.~12.3]{May1999}, we define the chain complex $\cI$
given by $\cI_0 = \langle [0], [1]\rangle$ and $\cI_1 = \langle [0,1]\rangle$ and
$\cI_k = 0$ for $k >0$. This is the cellular chain complex of the unit interval $I$ decomposed into two $0$-cells and one $1$-cell. A chain homotopy $h_b^\veps:f^\veps_b \simeq g^\veps_b$ corresponds to a chain map $h_b^\veps \colon C_*^{\rm cell}(X_b) \otimes \cI \rightarrow C_*^{\rm cell}(Y_b)$ such that $h_b^\veps(x,[0])= f^\veps_b(x)$
and $h_b^\veps(x,[1])=g^\veps_b(x)$ for all $x \in X_b$. Let $H_b^\veps(c,i) = F_b^\veps(c)$ for a cell $(c,i) \in X \times I$. By assumption $H^\veps \colon X \times I \rightrightarrows Y$ is an $\veps$-acyclic carrier. Note that $C_*^{\rm cell}(X_b) \otimes \cI \cong C_*^{\rm cell}(X_b \times I)$. Replicating the first part of the proof we can now extend any map $h_b^\veps \colon C_*^{\rm cell}(X_b) \otimes \cI_0
 \rightarrow C_*^{\rm cell}(Y_b)$ with the above properties to all cells of $X \times I$.
\end{proof}

\begin{defn}
Let $X_*, Y_* \in \RCWcpx$. We will call an $\veps$-acyclic carrier $I^{\veps}_X:X_*\rightrightarrows X_{*+\veps}$ carrying the standard shift $\Sigma^\veps X_*$ a \emph{shift} carrier. Suppose that there are $\veps$-acyclic carriers
\begin{align*}
 F^\veps:X_* \rightrightarrows Y_{*+\veps} \ ,\\
 G^\veps:Y_* \rightrightarrows X_{*+\veps} \ .
\end{align*}
together with shift carriers $I_X^{2\veps}$ and $I_Y^{2\veps}$.
We say that $X_*$ and $Y_*$ are \emph{$\veps$-acyclic equivalent}
whenever we have inclusions $G^\veps \circ F^\veps \subseteq I^{2\veps}_X$
and $F^\veps \circ G^\veps\subseteq I^{2\veps}_Y$.
\end{defn}

The motivation for the definition of $\veps$-acyclic equivalences is the following lemma:
\begin{prop}
\label{prop:FCWcpx-veps-interleaving}
Let $X_*$ and $Y_*$ be two tame elements from $\FCWcpx$
which are \emph{$\veps$-acyclic equivalent}.
Then $\PH(X_*)$ and $\PH(Y_*)$ are $\veps$-interleaved.
\end{prop}
\begin{proof}
By Prop.~\ref{prop:eps-acyclic-carriers} we know that there exist two chain maps
$f^\veps_*\colon C_*(X_*) \rightarrow C_{*}(Y_{*+\veps})$ and
$g^\veps \colon C_*(Y_*) \rightarrow C_{*}(X_{*+\veps})$
carried by $F^\veps$ and $G^\veps$ respectively. By
hypothesis the compositions $g^\veps \circ f^\veps$ and $f^\veps \circ g^\veps$
are carried by corresponding shift carriers
$I_{X}^{2\veps}$ and~$I_{Y}^{2\veps}$.
Thus, using the second part of Prop.~\ref{prop:eps-acyclic-carriers} we
obtain chain homotopies
$g^\veps \circ f^\veps \simeq \Sigma^{2\veps}C_*(X)$ and
$f^\veps \circ g^\veps \simeq \Sigma^{2\veps}C_*(Y)$.
Altogether, in homology these compositions are equal to the corresponding shifts,
and $\PH_*(X_*)$ and $\PH_*(Y_*)$ are $\veps$-interleaved.
\end{proof}

\begin{exmp}
\label{ex:vietoris-rips-complexes}
Consider two finite metric spaces $\bX$ and $\bY$. Let $d_H(\bX,\bY)$ be their Hausdorff distance and set $\veps = 2d_H(\bX,\bY)$.
Given a subcomplex $K\subseteq \VR(\bX)$, we denote its vertex set by $\bX(K)\subseteq \bX$. Likewise for a simplex $\sigma \in \VR(\bX)$, we write $\bX(\sigma)\subseteq \bX$ for the vertices spanning $\sigma$.
Define a carrier $F^\veps:\VR(\bX) \rightrightarrows \VR(\bY)$ by mapping a simplex $\sigma \in \VR(\bX)_a$ to
\[
    F^\veps(\sigma) = |\sup\{K \subseteq \VR(\bY)_{a+\veps} \mid d_{\rm H}(\bX(\sigma),\bY(K))\leq \veps/2\}|
\]
This is clearly semicontinuous. If $v_0, \dots, v_n$ are vertices in $F^\veps(\sigma)$, then by definition $\{v_0, \dots, v_n\}$ is an $n$-simplex in $F^\veps(\sigma)$. Therefore we have $F^\veps(\sigma) \simeq \Delta^N$ for some $N \in \bN$, which is acyclic. In particular, $F^\veps$ is an $\veps$-acyclic carrier. Interchanging the roles of $\bX$ and $\bY$ we also obtain an $\veps$-acyclic carrier $G^\veps:\VR(\bY) \rightrightarrows \VR(\bX)$. Similarly, we
define for a simplex $\sigma \in \VR(\bX)_a$ the shift carrier
\[
    I^{2\veps}_\bX(\sigma) = |\sup\{K \subseteq \VR(\bX)_{a+2\veps} \mid d_{\rm H}(\bX(\sigma),\bX(K))\leq \veps\}|
\]
Analogously one defines $I^{2\veps}_\bY$. Since $G^\veps \circ F^\veps \subseteq I^{2\veps}_\bX$ and
$F^\veps \circ G^\veps \subseteq I^{2\veps}_\bY$, Prop.~\ref{prop:FCWcpx-veps-interleaving}
implies that $\PH_*(\VR(\bX))$ and $\PH_*(\VR(\bY))$ are $\veps$-interleaved.
This is similar to the proof using \emph{correspondences},
see~\citep[Prop.~7.8, Sec.~7.3]{Oudot2015}.
\end{exmp}

\begin{exmp}
\label{ex:cubical-complexes}
Consider $\bR^N$ together with a $1$-Lipschitz function $f:\bR^N \rightarrow \bR$
with constant $\veps>0$. On the other hand consider the lattices $\bZ^N$
and $r\bZ^N + l$ for a pair of vectors $r,l \in \bR^N$ such that the coordinates
of $r$ satisfy $0 < r_i \leq 1$ for all $1 \leq i \leq N$. Then we take their
corresponding cubical complexes $\cC(\bZ^N)$ and $\cC(r\bZ^N + l)$ thought as
embedded in $\bR^N$. The function $f$ induces a natural filtration for these
cubical complexes: a vertex $v \in \cC(\bZ^N)$ is contained in
$\cC(\bZ^N)_{f(v)}$, while a cell $a \in \cC(\bZ^N)$ appears at the maximum
filtration value on its vertices.
There is a $\veps$-acyclic carrier $F^\veps:\cC(\bZ^N)\rightrightarrows\cC(r\bZ^N + l)$
sending each cell $a \in \cC(\bZ^N)$ to the smallest subcomplex $F^\veps(a)$
containing all $b \in \cC(r\bZ^N + l)$ such that $\overline{b}\cap a \neq \emptyset$.
In an analogous way the inverse acyclic carrier can be defined, and the
compositions $F^\veps \circ G^\veps$ and $G^\veps \circ F^\veps$ define the
shift carriers.
Thus, using Proposition~\ref{prop:FCWcpx-veps-interleaving}, one shows that
$\PH_*(\cC(\bZ^N))$ and $\PH_*(\cC(r\bZ^N + l))$ are $\veps$-interleaved.
\end{exmp}

Here an important assumption of Proposition~\ref{prop:eps-acyclic-carriers} is
that we are dealing with tame filtered CW-complexes. However, what if we considered
instead a pair of elements $X_*, Y_* \in \RCWcpx$? In this context, we notice
that given an $\veps$-acyclic carrier $F^\veps:X_*\to Y_*[\veps]$, it is not
necessarily true that the compositions
\[
Y(a+\veps \leq b + \veps)F_a^\veps(c) \mbox{ and }
F_b^\veps(X(a\leq b)(c))
\]
are still acyclic for all pairs $a\leq b$ from $\bbR$. Thus, whenever we talk
about $\veps$-acyclic carriers $F^\veps:X_*\to Y_*[\veps]$ in this context we will assume that
$F_b^\veps(X(a\leq b)(c))$ is acyclic for all pairs $a,b \in \bbR$ with $a \leq b$
and all cells $c \in X(a)$.
\begin{lem}
\label{lem:RCWcpx-veps-interleaving}
Let $X_*, Y_*\in \RCWcpx$ be a pair of elements such that both are $\veps$-acyclic
equivalent in the above sense. Then $d_I(\PH_*(X_*), \PH_*(Y_*))\leq \veps$.
\end{lem}
\begin{proof}
For each persistence value $a \in \bbR$, we use Theorem~\ref{thm:acyclic-carrier} twice
to obtain a pair of chain morphisms $f_a:C^\cell_a(X)\rightarrow C^\cell_{a+\veps}(Y)$ and $g_{a+\veps}: C^\cell_{a+\veps}(Y)\rightarrow C^\cell_{a+2\veps}(X)$. In a similar way
we obtain a pair of chain homotopies $g_{a+\veps}\circ f_a \simeq (\Sigma^{2\veps} C^\cell_*(X))_a$
and $f_{a+\veps} \circ g_a \simeq (\Sigma^{2\veps}C^\cell_*(Y))_a$ so that
we have equalities between the induced homology morphisms
$[g_{a+\veps}] \circ [f_a] = [(\Sigma^{2\veps} C^\cell_*(X))_a]$
and $[f_{a+\veps}] \circ [g_a] = [(\Sigma^{2\veps}C^\cell_*(Y))_a]$
for all $a \in \bbR$. Now, for a pair
of values $a \leq b$ from $\bbR$, it is not necessarily true that
$Y(a+\veps \leq b + \veps)\circ f_a = f_{b}\circ X(a\leq b)$. However, since
$Y(a+\veps \leq b + \veps)\circ f_a$ and $f_{b} \circ X(a\leq b)$ are both
included in $F_b^\veps(X(a\leq b)(c))$ by hypotheses, then by applying
Theorem~\ref{thm:acyclic-carrier} again there is a chain homotopy equivalence
$Y(a+\veps \leq b + \veps)\circ f_a \simeq f_{b}\circ X(a\leq b)$, which implies
%So that we obtain the commutativity equality
\[
[Y(a+\veps \leq b + \veps)]\circ [f_a] =  [f_{b}]\circ [X(a\leq b)]\ ,
\]
and we have defined a persistence morphism
$[f_*]:\PH_*(X_*)\rightarrow \PH_*(Y_*[\veps])$. Similarly we can also
put together the $g_a$ for all $a \in \bbR$ so that we obtain a morphism
$[g_*]:\PH_*(Y_*)\rightarrow \PH_*(X_*[\veps])$. This leads to the claimed $\veps$-interleaving.
\end{proof}
\begin{exmp}
Consider a point cloud $\bX$ and a Vietoris Rips complex $\VR_*(\bX)$ on top of it.
For this example, we will look at this applying an exponential transformation
$X_r = \VR_{\exp(r)}(\bX)$ for all $r \in \bbR$. We let $\cP_r$ be a partition of $\bX$ in such a way that for any partitioning set $U \in \cP_r$, for any pair of points
$p,q \in U$, we have that $d(p,q)\leq \exp(\veps)/2$. For a point $p \in \bX$, we will
denote by $\cP_r(p)$ the partitioning set from $\cP_r$ containing $p$. Then we define
a simplicial complex $Y_r$ over the set $\cP_r$, and where we include a simplex  $[\cP_r(p_0), \cP_r(p_1), \ldots, \cP_r(p_N)]$ whenever there exists a simplex  $[p_0, p_1, \ldots, p_N]$ in $\VR_{\exp(r)}(\bX)$; here we ignore degenerate simplices. It is easy to see that in general $Y_*$ is
not a filtered complex but rather a regularly filtered CW complex. We define the $\veps$-acyclic carrier $F^\veps_r:X_r\to Y_{r+\veps}$ by sending a simplex $[p_0, p_1, \ldots, p_N] \in X_r$ to the standard
simplex $\Delta^{[\cP_r(p_0), \cP_r(p_1), \ldots, \cP_r(p_N)]}$.
On the other way, we start from a simplex $[U_0, U_1, \ldots, U_M] \in Y_r$ and send it
to the subcomplex of $X_{r+\veps}$ given by the join $\Delta(U_0)*\Delta(U_1)*\cdots * \Delta(U_M)$, where by $\Delta(U_i)$ we mean the standard simplex with vertices on the elements from $U_i \in \cP_r$.
These assignements lead to a $\veps$-acyclic equivalence between $X_r$ and $Y_r$, where the shift carriers are given by composition.
Now, by Lemma~\ref{lem:RCWcpx-veps-interleaving} there is a
$\veps$-interleaving between $\PH_*(X)$ and $\PH_*(Y)$.
\end{exmp}

\begin{rem}
Notice that our notion of acyclicity is different from that in~\citep{Cavanna2019} and~\citep{GovSkr2018}.
In~\citep{GovSkr2018} a filtered complex $K_*$ is called $\veps$-acyclic whenever
the induced homology maps $H_*(K_r)\rightarrow H_*(K_{r+\veps})$ vanish for all
$r \in \bR$. In this case one can still (trivially) define acyclic carriers
between $*$ and $K_*$. The problem arises when  defining the shift carrier
$I_K^{A\veps}$ for some constant $A>0$, which does not exist in general.
One can however, adapt the proof of Proposition~\ref{prop:eps-acyclic-carriers} so that there is a chain morphism $\psi^{\veps(\dimn(K_r)+1)}:C_*(K_r)\to C_*(K_{r+\veps(\dimn(K_r)+1)})$; and that this coincides up to chain homotopy with the composition through $C_*(*)$. One does this by following the same proof as in Proposition~\ref{prop:eps-acyclic-carriers}, but increasing the filtration value by $\veps$ each time we assume that some cycle lies in an acyclic carrier. Thus, if we have $\dimn(K)= \sup_{r\in \bR}(\dimn(K_r)) < \infty$, then one could say that there is an \emph{$\veps(\dimn(K)+1)$-approximate} chain homotopy equivalence between $C(*)$ and $C(K_*)$.
\end{rem}

\section{Interleaving Geometric Realizations}
\label{sec:intBlowCpx}

Next, we focus on acyclic carrier equivalences between a pair of diagrams $\cD, \cL \in \RRDiag{K}$. We start by taking $\veps$-acyclic carriers $F^\veps_\sigma :\cD(\sigma) \rightrightarrows \cL(\sigma)$ for all $\sigma \in K$ which have to be compatible in the following sense: For any pair $\tau \preceq \sigma$ and any cell $c \in \cD(\sigma)$, there is an inclusion
\begin{equation}
\label{eq:compatible-inclusion}
\cL(\tau \preceq \sigma)(F^\veps_\sigma(c)) \subseteq
F^\veps_\tau(\cD(\tau \preceq \sigma)(c))
\end{equation}
and we assume in addition that $F^\veps_\tau(\cD(\tau \preceq \sigma)\Sigma^r\cD(\sigma)(c))$ is acyclic for all $r\geq 0$.
This compatibility leads to `local' diagrams of spaces. That is,
given a pair of values $a \in \bbR$ and $r \geq 0$ and a cell $c \in \cD(\sigma)_a$, we
consider an object $F^{r, \veps}_{\sigma \times c} \in \RDiag{\Delta^\sigma}$.
It is given by the space $F^{r, \veps}_{\sigma \times c}(\tau) = F_\tau^\veps
\big( \cD(\tau\preceq \sigma)\Sigma^r \cD(\sigma)(c)\big)$ for all faces $\tau \preceq \sigma$.
For any sequence $\rho \preceq \tau \preceq \sigma$ in $K$, there are morphisms
in $F^{r, \veps}_{\sigma \times c}$ given by restricting morphisms from $\cL$
$$
\begin{tikzcd}
\tau \arrow[r] &
F^{r, \veps}_{\sigma \times c}(\tau) \ar[d] \ar[r, equal] &
F^\veps_\tau\big(\cD(\tau \preceq \sigma)\Sigma^r \cD(\sigma)(c)\big)
\arrow[d, "\cL(\rho \preceq \tau)"] \\
\rho \arrow[r] \arrow[u, "\preceq"] &
F^{r, \veps}_{\sigma \times c}(\rho) \ar[r, equal] &
F^\veps_\rho\big(\cD(\rho \preceq \sigma)\Sigma^r \cD(\sigma)(c)\big)\ .
\end{tikzcd}
$$
Using condition~(\ref{eq:compatible-inclusion}) repeatedly on the
cells from $L = \cD(\tau \preceq \sigma)\Sigma^r \cD(\sigma)(c)$, we see that we have an inclusion
\[
\cL(\rho \preceq \tau)\big( F^\veps_\tau (L)\big) \subseteq
F^\veps_\sigma \big( \cD(\rho \preceq \tau)(L)\big)\ .
\]
Thus the diagram $F^{r, \veps}_{\sigma \times c}$ is indeed well defined, and we may consider the geometric realization $\Delta F^{r, \veps}_{\sigma \times c}$. By hypothesis
each $F^{r, \veps}_{\sigma \times c}(\tau)$ is acyclic for all $\tau \preceq \sigma$,
so that the first page of the spectral sequence
$E^*_{p,q}(F^{r, \veps}_{\sigma \times c})\Rightarrow \Ho_{p+q}(\Delta F^{r, \veps}_{\sigma \times c})$
is equal to
\[
E^1_{p,q}(F^{r, \veps}_{\sigma \times c}) =
\bigoplus_{\tau \in (\Delta^\sigma)^p} \Ho_q(F^{r, \veps}_{\sigma \times c}(\tau))
= \begin{cases}
\bigoplus_{\tau \in (\Delta^\sigma)^p} \bF \mbox{ if } q=0, \\
0 \mbox{ otherwise.}
\end{cases}
\]
where $\bF$ denotes our coefficient field for homology.
In fact, computing the homology with respect to the horizontal differentials on the
first page corresponds to computing the homology of $\Delta^\sigma$.
Thus, $E^2_{p,q}(F^{r, \veps}_{\sigma \times c})$ is zero everywhere except at
$p=q=0$ where it is equal to $\bF$. Thus, the spectral sequence collapses on
the second page, and $\Delta F^{r, \veps}_{\sigma \times c}$ is acyclic. We will
use the notation $F^\veps_{\sigma \times c} = F^{0, \veps}_{\sigma \times c}$.

\begin{defn}
Let $\cD$ and $\cL$ be two diagrams in $\RRDiag{K}$. Suppose that there are $\veps$-acyclic
carriers $F^\veps_\sigma:\cD(\sigma)\rightrightarrows \cL(\sigma)$ for all $\sigma \in K$, and such that
\[
\cL(\tau \preceq \sigma)\big( F^\veps_\sigma (c)\big) \subseteq
F^\veps_\sigma \big( \cD(\tau \preceq \sigma)(c)\big)\
\]
for all $c \in \cD(\sigma)$
and in addition $F^\veps_\tau(\cD(\tau \preceq \sigma)\Sigma^r\cD(\sigma)(c))$ is acyclic for all $r\geq 0$. Then we say that the set $\{F^\veps_\sigma\}_{\sigma \in K}$ is a $(\veps, K)$-acyclic carrier between $\cD$ and $\cL$.  We denote this by $F^\veps:\cD \rightrightarrows \cL$.
\end{defn}

\begin{thm}
\label{thm:geometric-realization}
Let $\cD$ and $\cL$ be two diagrams in $\RRDiag{K}$. Suppose that there are $(\veps, K)$-acyclic carriers $F^\veps:\cD\rightrightarrows \cL$ and $G^\veps:\cL\rightrightarrows \cD$, together with a pair of shift $(\veps, K)$-acyclic carriers $I^{2\veps}_\cD:\cD\rightrightarrows \cD$ and $I^{2\veps}_\cL:\cL\rightrightarrows \cL$, and such that these restrict to acyclic equivalences
\[
G^\veps_\tau \circ F^\veps_\tau \subseteq (I^{2\veps}_\cD)_\tau \mbox{ and } F^\veps_\tau \circ G^\veps_\tau \subseteq (I^{2\veps}_\cL)_\tau
\]
for each simplex $\tau \in K$.
Then there is an $\veps$-acyclic equivalence
$F^\veps:\Delta\cD\rightrightarrows \Delta\cL$ which preserves filtrations. That is,
there are $\veps$-acyclic equivalences $F^pF^\veps:F^p \Delta \cD \rightrightarrows F^p\Delta \cL$ for all $p\geq 0$.
\end{thm}

\begin{proof}
Let $\sigma\times c \in \Delta\cD$ be a cell, where $c$ is an $m$-cell
in $\cD(\sigma)$. Define an acyclic carrier
$F^\veps: \Delta \cD \rightrightarrows \Delta \cL$ by sending
$\sigma \times c$ to the acyclic carrier $\Delta F^\veps_{\sigma\times c}$, which is a subcomplex of $\Delta\cL$. Let us first check semicontinuity. For any pair of cells
$\tau \times a \preceq \sigma\times c$ in $\Delta\cD$, the cell $a$ is contained in
the subcomplex $\overline{\cD(\tau \preceq \sigma)(c)}$, and by continuity of $\cD(\rho \preceq \tau)$
we have that $\cD(\rho \preceq \tau)(a) \subseteq \overline{\cD(\rho \preceq \sigma)(c)}$.
Thus there are inclusions
$$
F^\veps_{\rho}(\cD(\rho \preceq \tau)(a)) \subseteq
F^\veps_{\rho}(\overline{\cD(\rho \preceq \sigma)(c)}) =
F^\veps_{\rho}(\cD(\rho \preceq \sigma)(c))
$$
for all $\rho \preceq \tau$. More concisely,  $F_{\tau \times a}^\veps(\rho) \subseteq F_{\sigma \times c}^\veps(\rho) $ for all $\rho \preceq \tau$. As a consequence $\Delta F_{\tau \times a}^\veps \subseteq \Delta F_{\sigma \times c}^\veps$ and semicontinuity holds.

Next, notice that $F^\veps \big(\Sigma^r \Delta \cD (\sigma \times c)\big)=
F^\veps \big(\sigma \times \Sigma^r \cD(\sigma) (c)\big) =
\Delta F^{r, \veps}_{\sigma \times c}$ which is an acyclic carrier. In order for
$F^\veps$ to be an $\veps$-acyclic carrier, it remains to show the inclusion
$\Sigma^r \Delta \cL \circ F^\veps \subseteq F^\veps \circ \Sigma^r \Delta \cD$
for all $r \geq 0$. For this, take $\sigma \times c \in \Delta \cD$ and see that
\begin{multline}
\nonumber
\Sigma^r \Delta \cL \circ F^\veps (\sigma \times c)
= \Sigma^r \Delta \cL \bigg( \bigcup_{\tau \preceq \sigma} \tau \times F^\veps_\tau\big(\cD(\tau \preceq \sigma)(c)\big)\bigg) \\
= \bigcup_{\tau \preceq \sigma} \tau \times \Sigma^r \cL(\tau)\big( F^\veps_\tau\big(\cD(\tau \preceq \sigma)(c)\big)\big)
\subseteq \bigcup_{\tau \preceq \sigma} \tau \times F^\veps_\tau\big(\Sigma^r\cD(\tau)\cD(\tau \preceq \sigma)(c)\big) \\
= \bigcup_{\tau \preceq \sigma} \tau \times F^\veps_\tau\big(\cD(\tau \preceq \sigma)\Sigma^r\cD(\sigma)(c)\big)
= F^\veps\big(\sigma \times \Sigma^r \cD(\sigma)(c)\big) = F^\veps \circ \Sigma^r \Delta \cD (\sigma \times c)\ .
\end{multline}

Similarly, one can define an $\veps$-acyclic carrier $G^\veps: \Delta \cL \rightrightarrows \Delta \cD$ sending $\sigma \times c \in \Delta \cL$ to $\Delta G^\veps_{\sigma \times c}$.  In addition, we define respective shift $\veps$-acyclic carriers $I^{2\veps}_\cD:\Delta \cD \rightrightarrows \Delta \cD$ and $I^{2\veps}_\cL:\Delta \cL \rightrightarrows \Delta \cL$, sending respectively $\sigma \times c \in \Delta \cD$ to $\Delta (I^{2\veps}_\cD)_{\sigma \times c}$ and  $\tau \times a \in \Delta \cL$ to $\Delta (I^{2\veps}_\cL)_{\tau \times a}$.
Then we have
\begin{multline}
\nonumber
G^\veps \circ F^\veps (\sigma \times c) = G^\veps(\Delta F^\veps_{\sigma \times c}) =
G^\veps \Big(
\bigcup_{\tau \preceq \sigma}
\tau\times F^\veps_\tau\big(\cD(\tau \preceq  \sigma)(c)\big) \Big) \\
= \bigcup_{\rho \preceq  \tau \preceq \sigma} \rho \times
G^\veps_{\rho}\Big(\cL(\rho \preceq  \tau)
F_{\tau}^\veps \big(\cD(\tau \preceq \sigma)(c)\big)\Big)  \\
\subseteq \bigcup_{\rho\preceq \sigma} \rho \times
G^\veps_{\rho} F^\veps_{\rho} \big(\cD(\rho \preceq \sigma)(c)\big)
\subseteq \Delta (I^{2\veps}_\cD)_{\sigma \times c} = I^{2\veps}_\cD(\sigma \times c),
\end{multline}
where we have used the commutativity condition and equivalence of
$F^\veps_{\rho}$ and $G^\veps_{\rho}$.
Consequently $G^\veps \circ F^\veps \subseteq I^{2\veps}_\cD$;
the other inclusion $F^\veps \circ G^\veps \subseteq I^{2\veps}_\cL$
follows by symmetry. Altogether, we have obtained an $\veps$-equivalence
$F^\veps:\Delta \cD \rightrightarrows \Delta \cL$. Finally, notice that for all $p\geq 0$ and for each cell $\sigma \times c \in F^p\Delta \cD$, its carrier $\Delta F^\veps_{\sigma \times c}$
is contained in $F^p\Delta \cD$ and so it preserves filtration.
The same follows for the other acyclic carriers.
\end{proof}

Let $X$ be a filtered simplicial complex, together with a cover $\cU$ by filtered
subcomplexes. Recall the definitions of the diagrams $X^\cU$ and $\pi_0^\cU$
over $N_\cU$ from example~\ref{ex:cover-spaces}. Consider the case when
$d_I\big(\PH_*(X^\cU(\sigma)), \PH_*(\pi^\cU_0(\sigma))\big) \leq \veps$ for all
$\sigma \in N_\cU$. This example has been of interest before, see for
example~\citep{GovSkr2018} or~\citep{Cavanna2019}. As mentioned in the remark at the end of Section~\ref{sec:eps-carriers}, our notion of $\veps$-acyclicity is much stronger than that from~\citep{GovSkr2018}. This is why we obtain a result closer to the \emph{Persistence Nerve Theorem} from~\citep{ChazOudo2008} than to the \emph{Approximate Nerve Theorem} from~\citep{GovSkr2018}.

\begin{cor}[Strong Approximate Multinerve Theorem]
\label{thm:strong-multinerve}
Consider a diagram $\cD$ in $\FRDiag{K}$.
Assume that there are is a $(\veps, K)$-acyclic carrier $F^\veps: \pi_0 \cD \rightrightarrows \cD$.
Then, there is a $\veps$-acyclic equivalence
$F^\veps:\MN(\cD) \rightrightarrows \Delta \cD$.
Consequently,
$$
d_I(\PH_*(\MN(\cD)), \PH_*(\Delta \cD)) \leq \veps\ .
$$
\end{cor}
\begin{proof}
Notice that $\pi_0\cD$ is a well defined element from $\RRDiag{K}$ and there is an obvious choice for the $(\veps,K)$-acyclic carrier $G^\veps: \cD \rightrightarrows \pi_0 \cD$ where we send cells to their corresponding connected component classes.
The compatibility condition
$$
\pi_0 (\cD(\tau \preceq \sigma))(G^\veps(\cD(\sigma)))
\subseteq G^\veps(\cD(\tau))
$$
also follows.
The shift $(2\veps,K)$-carrier $I^{2\veps}_{\pi_0\cD}$ sends points to points, while the other $I^{2\veps}_{\cD}$ is defined as the composition $F^\veps \circ G^\veps$, which can be checked to define a $(2\veps,K)$-acyclic carrier.
Altogether, we can use
Proposition~\ref{thm:geometric-realization} and there exists a $\veps$-acyclic
equivalence $F^\veps:\MN(\cD)\rightrightarrows \Delta(\cD)$.
\end{proof}

\begin{exmp}
Consider a point cloud $\bX$ together with a partition $\cP$. Assume that the
$(\VR_*(\bX), \cP)$-join diagram $\cJ^{\VR_*(\bX)}_\cP$ is such that there are
compatible $\veps$-acyclic equivalences $\pi_0(\cJ^{\VR_*(\bX)}_\cP(\sigma))
\rightrightarrows \cJ^{\VR_*(\bX)}_\cP(\sigma)$ for all $\sigma \in \Delta^{P}$.
Then there is an $\veps$-acyclic equivalence
$\Delta\pi_0\big(\cJ^{\VR_*(\bX)}_\cP\big)\rightrightarrows \Delta\cJ^{\VR_*(\bX)}_\cP$
so that
\[
d_I\big(\PH_*(\MN(\cJ^{\VR_*(\bX)}_\cP)), \PH_*(\VR_*(\bX))\big)\leq \veps \ .
\]
\end{exmp}

Acyclic carriers have been used in~\citep{Kaczynski2004} and in~\citep{Nanda2012} for
approximating continuous morphisms by means of simplicial maps. Here we have
used the same tools to obtain an approximate homotopy colimit theorem.
The acyclic carrier theorem is an instance of the more general acyclic Model theorem.
An interesting future research direction would be to see how that general result can bring
new insights into applied topology.

\section{Interleaving Spectral Sequences}
\label{sec:ss-interleavings}
\begin{defn}
Let $\cA$ and $\cB$ from $\SpSq$. A $n$-spectral sequence morphism $f:\cA \to \cB$ is a spectral
sequence morphism $f:\cA \rightarrow \cB$ which is defined from page~$n$.
\end{defn}
\begin{defn}
Given two objects $\cA$ and $\cB$ in \PSpSq.
We say that $\cA$ and $\cB$ are $(\varepsilon, n)$-interleaved whenever there exist
two $n$-morphisms $\psi : \cA \rightarrow \cB [\varepsilon]$ and
$\varphi : \cB \rightarrow \cA [\varepsilon]$ such that the following diagram commutes

\begin{equation}
\label{diag:interleaving}
\begin{tikzcd}
\cA \ar[dr, "\psi", near start, swap] \ar[d, "\Sigma^\veps \cA", swap] &
\cB \ar[dl, "\varphi", near start, crossing over] \ar[d, "\Sigma^\veps \cB"] \\
\cA {[\varepsilon]} \ar[dr, "\psi {[\varepsilon]}", swap, near start]
	\ar[d, "\Sigma^\veps \cA {[\varepsilon]}", swap] &
\cB {[\varepsilon]} \ar[dl, "\varphi {[\varepsilon]}", near start, crossing over]
	\ar[d, "\Sigma^\veps \cB {[\varepsilon]}"] \\
\cA {[2\varepsilon]} & \cB {[2\varepsilon]}
\end{tikzcd}
\end{equation}
for all pages $r \geq n$.
This interleaving defines a pseudometric in \PSpSq
$$
d_I^n\big( \cA, \cB \big) \coloneqq
\inf \left\{ \varepsilon \mid \mbox{$\cA$ and $\cB$ are $(\varepsilon, n)$-interleaved}\right\}\ .
$$
\end{defn}

\begin{prop}
\label{prop:inequalities-pages}
Suppose that $\cA$ and $\cB$ are $(\varepsilon, n)$-interleaved. Then these are
$(\varepsilon, m)$-interleaved for all $m \geq n$. In particular, we have that
$$
d_I^m \big( \cA, \cB \big) \leq d_I^n \big( \cA, \cB\big)
$$
for any pair of integers $m \geq n$.
\end{prop}
\begin{proof}
Follows directly from the definitions.
\end{proof}

We start now by considering Mayer-Vietoris spectral sequences. Under some
conditions which are a special case of Theorem~\ref{thm:geometric-realization},
one can obtain one page stability. In fact this stability is due to morphisms
directly defined on the underlying double complexes, which is a very
strong property.

\begin{prop}
\label{prop:stability-ss}
Let $(X, \cU)$ and $(Y, \cV)$ be two tame elements in $\FCWcpx$ together with covers by
subcomplexes, both having the same finite nerve $K=N_\cU=N_\cV$.
Suppose that there are $(\veps, K)$-acyclic carriers $F^\veps:X^\cU\rightrightarrows Y^\cV$ and $G^\veps:Y^\cV\rightrightarrows X^\cU$, together with a pair of shift $(\veps, K)$-acyclic carriers $I^{2\veps}_{X^\cU}:X^\cU\rightrightarrows X^\cU$ and $I^{2\veps}_{Y^\cV}:Y^\cV\rightrightarrows Y^\cV$, and such that these restrict to acyclic equivalences
\[
G^\veps_\tau \circ F^\veps_\tau \subseteq (I^{2\veps}_{X^\cU})_\tau \mbox{ and } F^\veps_\tau \circ G^\veps_\tau \subseteq (I^{2\veps}_{Y^\cV})_\tau
\]
for each simplex $\tau \in K$.
Then there are a pair of double complex morphisms $\phi^\veps:C_{*,*}(X, \cU)\to C_{*,*}(Y, \cV)[\veps]$
and $\psi^\veps:C_{*,*}(Y, \cV)\to C_{*,*}(X, \cU)[\veps]$ inducing a first page interleaving
between $E^*_{*, *}(X, \cU)$ and $E^*_{*, *}(Y, \cV))$.
\end{prop}
\begin{proof}
Unpacking the definitions this means we have to give chain homomorphisms
\begin{align*}
    (\phi_\sigma^\veps)_r & \colon C_*(X^\cU(\sigma)_r) \to C_*(Y^\cV(\sigma)_{r+\veps}) \ ,\\
    (\psi_\sigma^\veps)_r & \colon C_*(Y^\cV(\sigma)_r) \to C_*(X^\cU(\sigma)_{r+\veps})
\end{align*}
that are natural in $\sigma \in K$ and in $r \in \bbR$. Since $K$ is a poset
category, these can be constructed inductively as follows: As in
Prop.~\ref{prop:eps-acyclic-carriers} we may define $\phi_\sigma^\veps$ on all
simplices $\sigma \in K$ of dimension $\dimnp{\sigma} = \dimnp{K}$. Note that
$(\phi_\sigma^\veps)_r$ is carried by $(F^\veps_\sigma)_r$ for all $r \in \bbR$.
Assume by induction that $\phi_\tau^\veps$ are defined and carried by
$F^\veps_\tau$ for all $\tau \in K$ with $n \leq \dimnp{\tau} \leq \dimnp{K}$ in
such a way that for all cofaces $\tau \preceq \sigma$ the naturality condition
$\phi_\tau^\veps \circ X^\cU(\tau \prec \sigma) = Y^\cV(\tau \prec \sigma)[\veps]
\circ \phi_\sigma^\veps$ holds. Now let $\tau \in K$ have dimension
$\dimnp{\tau} = n-1 \geq 0$. The naturality condition on the simplices fixes
$\phi_\tau^\veps$ on the filtered subcomplex $X^\tau =
\bigcup_{\tau \prec \sigma} \img \big(X^\cU(\tau \prec \sigma)\big)$, where the
union is taken over all cofaces $\sigma$ of $\tau$. Here notice that we can
assume that $\phi^\veps_\tau$ is well defined since the previous choices of
$\phi^\veps_\sigma$ for all cofaces $\tau \prec \sigma$ are consistent due to
the fact that for each cell $c \in X_\tau$ there exists a unique maximal simplex $\sigma \in N_\cU$
such that $c \in X_\sigma$.
In addition, notice that by hypotheses $Y^\cV(\tau \prec \sigma)((F_\sigma^\veps)(c))
\subseteq F_\tau^\veps(X^\cU(\tau \prec \sigma)(c))$ for all $a \in \bbR$ and
$c \in X^\cU(\sigma)$, so that our definition of $\phi^\veps_\tau$ in $X^\tau$ is
indeed carried by $F^\veps_\tau$. We then proceed as in Prop.~\ref{prop:eps-acyclic-carriers}
to define $(\phi_\tau^\veps)_a$ on all simplices in the subset $X^\cU(\tau)_a
\setminus X^\tau_a$ for all $a\in \bbR$. The resulting chain map
$(\phi_{\tau}^\veps)_a$ is carried by $(F^\veps_\tau)_a$ for all $a \in \bbR$.
Since $X^\cU$ is tame, we only need finitely many steps to obtain a morphism
$\phi_\tau^\veps: C_*(X^\cU(\tau)) \rightarrow C_*(Y^\cV(\tau)[\veps])$ that
satisfies the induction hypotheses.

Thus, we obtain double complex morphisms $\phi_{p,q}^\veps:C_{p,q}(X, \cU)
\to C_{p,q}(Y, \cV)[\veps]$ for all $p,q \geq 0$ by adding up our defined
local morphisms
\[
\phi_{p,q}^\veps: \bigoplus_{\sigma \in K^p} \phi^\veps_\sigma :
\bigoplus_{\sigma \in K^p} C_q(X^\cU(\sigma)) \longrightarrow
\bigoplus_{\sigma \in K^p} C_q(Y^\cV(\sigma))[\veps]\ .
\]
Notice that $\phi_{p,q}^\veps$ commute both with horizontal
and vertical differentials since we assumed that each $\phi^\veps_\sigma$ is a
chain morphism and these satisfy a naturality condition with respect to $K$. Thus,
this double complex morphism induces a spectral sequence morphism
$\phi^\veps_{p,q}:E^*_{p,q}(X^\cU)\to E^*_{p,q}(Y^\cV)[\veps]$. By doing the same
construction, we can obtain local chain morphisms
$\psi_\sigma^\veps:C_*(Y^\cV(\sigma))\to C_*(X^\cU(\sigma))[\veps]$ so that
by Proposition~\ref{prop:eps-acyclic-carriers}  we have equalities
$[\psi_\sigma^\veps] \circ [\phi_\sigma^\veps] = [\Sigma^{2\veps}C_*(X^\cU(\sigma))]$
and also
$[\phi_\sigma^\veps] \circ [\psi_\sigma^\veps] = [\Sigma^{2\veps}C_*(Y^\cV(\sigma))]$
for all $\sigma \in K$.
Then we can construct a double complex morphism $\psi_{p,q}^\veps:C_{p,q}(Y, \cV)\to C_{p,q}(X, \cU)[\veps]$
inducing an ``inverse'' spectral sequence morphism $\psi^\veps_{p,q}:E_{p,q}^*(Y, \cV)\to E_{p,q}^*(X, \cU)[\veps]$.
These are such that from the first page, $\phi^\veps_{*,*}$ and $\psi^\veps_{*,*}$
form a $(\veps, 1)$-interleaving of spectral sequences.
\end{proof}

Notice that the proof of Proposition~\ref{prop:stability-ss} relies heavily on the fact
that the diagrams we are considering come from a cover. This allows us to define a pair of double complex morphisms that are compatible along the common indexing nerve.
However, in Theorem~\ref{thm:geometric-realization} we observed that, under some conditions, the geometric realizations of regularly filtered regular diagrams are stable. Does this stability carry over to the associated spectral sequences? The next
theorem shows that this is indeed the case.

\begin{thm}
  \label{thm:stability-ss}
Let $\cD$ and $\cL$ be two diagrams in $\RRDiag{K}$. Suppose that there are $(\veps, K)$-acyclic carriers $F^\veps:\cD\rightrightarrows \cL$ and $G^\veps:\cL\rightrightarrows \cD$, together with a pair of shift $(\veps, K)$-acyclic carriers $I^{2\veps}_\cD:\cD\rightrightarrows \cD$ and $I^{2\veps}_\cL:\cL\rightrightarrows \cL$, and such that these restrict to acyclic equivalences
\[
G^\veps_\tau \circ F^\veps_\tau \subseteq (I^{2\veps}_\cD)_\tau \mbox{ and } F^\veps_\tau \circ G^\veps_\tau \subseteq (I^{2\veps}_\cL)_\tau
\]
for each simplex $\tau \in K$.
  Then
  $$
  d_I^1(E(\cD, K), E(\cL, K)) \leq \veps\ .
  $$
\end{thm}
\begin{proof}
Recall from Theorem~\ref{thm:geometric-realization} that there is a filtration-preserving acyclic carrier
$F^\veps:\Delta_K \cD \rightrightarrows \Delta_K\cL[\veps]$.
This implies that there
are chain complex morphisms $f^\veps_a: C_*(\Delta \cD)_r\to C_*(\Delta \cL)_{a+\veps}$ which
respect filtrations in the sense that
$f^\veps_{r}(F^p C_*(\Delta \cD)_r) \subseteq F^p C_*(\Delta \cL)_{r+\veps}$  for all $p\geq 0$.
By Lemma~\ref{lem:iso-geom-total} this defines a morphism
$f^\veps_r:S^\Tot_*(\cD)_r \to S^\Tot_*(\cL)_{r+\veps}$ which respects filtrations.
Altogether we deduce that $f^\veps_r$ determines a morphism
of spectral sequences $f^\veps_r:E^*_{p,q}(\cD)_r\rightarrow E^*_{p,q}(\cL)_{r+\veps}$.
Similarly as in Lemma~\ref{lem:RCWcpx-veps-interleaving} the commutativity
\begin{equation}
\label{eq:commutativity-ssmorphism}
\Sigma^\veps E^*_{p,q}(\cL)_{r+\veps}\circ f^\veps_r = f^\veps_{r+\veps}\circ \Sigma^\veps E^*_{p,q}(\cD)_r
\end{equation}
does not need to hold for all $r \in \bbR$. However, recall
from Theorem~\ref{thm:geometric-realization} that there is an inclusion
$\Sigma^r \Delta \cL \circ F^\veps \subseteq F^\veps \circ \Sigma^r \Delta \cD$ where
the superset is acyclic, so that $\Sigma^r C_*(\Delta \cD)_{a+\veps}\circ f^\veps_a$ and  $f^\veps_{a+r} \circ \Sigma^r C_*(\Delta \cD)_a$
 are both carried by the filtration preserving acyclic carrier $F^\veps \circ \Sigma^r \Delta \cD$.
 This implies that there
exist chain homotopies $h^\veps_r:C_n(\Delta \cD)_r\to C_{n+1}(\Delta \cD)_{r+\veps}$
which respect filtrations and such that
\[
f^\veps_{a+r} \circ \Sigma^r C_*(\Delta \cD)_a
- \Sigma^r C_*(\Delta \cD)_{a+\veps}\circ f^\veps_a
= \delta^\Delta \circ h^\veps_a + h^\veps_a \circ \delta^\Delta \ .
\]
for all $a \in \bbR$ and all $r\geq 0$. Recall that the zero page terms are
given as quotients on successive filtration terms
$E^0_{p,q}(\cD)_a= F^p S^\Tot_{p+q}(\cD)_a/F^{p-1}S^\Tot_{p+q}(\cD)_a$, for
all $a \in \bbR$ and all integers $p,q \geq 0$. Thus, by
Lemma~\ref{lem:iso-geom-total} these chain homotopies carry over to $S^\Tot_*(\cD)_a$
and the commutativity relation from equation~(\ref{eq:commutativity-ssmorphism}) holds
from the first page onwards.

Similarly, we can define spectral sequence morphisms
$g^\veps_r:E^*_{p,q}(\cD)_a \rightarrow E^*_{p,q}(\cD)_{a+\veps}$ for all
$a \in \bbR$ which commute with the shift morphisms from the first page.
Also, by inspecting the shift carriers, we can obtain equalities of $1$-spectral
sequence morphisms
$g^\veps_{r+\veps} \circ f^\veps_r = \Sigma^{2\veps} E^*_{p,q}(\cD)_r$ and also
$f^\veps_{r+\veps} \circ g^\veps_r = \Sigma^{2\veps} E^*_{p,q}(\cL)_r$ for all
$r \in \bbR$, and the result follows.
\end{proof}

\begin{exmp}
Consider a pair of point clouds $\bX, \bY \in \bR^N$, together with partitions
$\cP$ and $\cQ$ for $\bX$ and $\bY$ respectively. Also, assume that there
is an isomorphism $\phi:\Delta^\cP\rightarrow \Delta^\cQ$ such that
$d_H(\bX\cap V, \bY\cap \phi(V)) < \veps$ for all $V \in \cP$. As defined in
example~\ref{ex:vietoris-rips-complexes} there are $\veps$-acyclic carrier
equivalences $F_V^\veps:\VR(\bX\cap V)\rightrightarrows \VR(\bY\cap V)$ for all $V \in \cU$.
For any $\sigma \in \Delta^\cP$, one can define $\veps$-acyclic equivalences $F_\sigma^\veps:\cJ^{\VR_*(\bX)}_\cP(\sigma)\rightrightarrows \cJ^{\VR_*(\bY)}_\cQ(\sigma)$ by
sending a cell $\prod_{V\in \sigma} \tau_V \in \cJ^{\VR_*(\bX)}_\cP(\sigma)$ to
$\prod_{V\in \sigma}F^\veps_V(\tau_V) \in \cJ^{\VR_*(\bY)}_\cQ(\sigma)$.
These are compatible and $d_I^1(E(\cJ^{\VR_*(\bX)}_\cP, \Delta^\cP), E(\cJ^{\VR_*(\bY)}_\cQ, \Delta^\cQ))\leq \veps$.
\end{exmp}

\section{Interleavings with respect to different covers}
\label{sec:interleavings-different-covers}

\subsection{Refinement Induced Interleavings}

In the previous sections we considered general diagrams in $\FRDiag{K}$ for some simplicial complex $K$. We will now focus on the situation where we have a filtered simplicial complex X together with a cover $\cU$ by filtered subcomplexes, which provides a diagram $X^\cU:N_\cU\to \FCWcpx$.  The associated spectral sequence will be denoted by $E^*_{*,*}(X,\cU)$,
as done at the start of section~\ref{sec:spectral_sequences}. We want to measure how $E^*_{*,*}(X,\cU)$ changes depending on $\cU$ and follow ideas from~\citep{Serre1955} to achieve this.
First we consider a refinement $\cV \prec \cU$, which means that for all
$V \in \cV$, there exists $U \in \cU$ such that $V \subseteq U$.
In particular, one can choose a morphism
$\rho^{\cU,\cV}:N_\cV \rightarrow N_\cU$ such that $\cV_{\sigma} \subseteq \cU_{\rho\sigma}$ for
all $\sigma \in N_\cV$. This choice is of course not necessarily unique.
We would like to compare the Mayer-Vietoris
spectral sequences of
both covers.
For this, we recall the definition of the \cech~chain complex
outlined in the introduction of~\citep{TorrasCasas2019}, which leads to the following isomorphism on the
terms from the $0$-page
\begin{equation}
\label{eq:cech-alternating}
E_{p,q}^0(X,\cU) \coloneqq \bigoplus_{\sigma \in N^p_\cU}C^\cell_q(\cU_\sigma)\simeq \bigoplus_{s \in X^q} f^{\sigma(s, \,\cU)}_*\Big(\,C^\cell_p\big(\,\Delta^{\sigma(s,\,\cU)}\,\big)\,\Big)\ .
\end{equation}
Here, $\sigma(a,\cU)$ is the simplex of maximal dimension in $N_\cU$ such that
$a \in X \cap \cU_{\sigma(a,\cU)}$, and $f^{\sigma(a, \cU)} \colon \Delta^{\sigma(a, \cU)}\hookrightarrow N_\cU$ denotes the inclusion.
The isomorphism in~(\ref{eq:cech-alternating}) is given by sending a generator $(a)_\sigma \in \bigoplus_{\sigma \in N^p_\cU}C^\cell_q(\cU_\sigma)$ to its transpose $(\sigma)_a$, for all cells $a \in X$ and all $\sigma \in N_\cU$.

Returning to a refinement $\cV\prec \cU$ and a morphism
$\rho^{\cU,\cV}: N_\cV \rightarrow N_\cU$, there is an induced double complex morphism
$\rho^{\cU,\cV}_{p,q}: C_{p,q}(X, \cU)\rightarrow C_{p,q}(X, \cV)$ given by
$$
\rho^{\cU,\cV}_{p,q}((\sigma)_a) =
\begin{cases}
  (\rho^{\cU,\cV}\sigma)_a & \mbox{if } \dim(\rho^{\cU,\cV}\sigma)=p, \\
  0  & \mbox{otherwise,}
\end{cases}
$$
for all generators $(\sigma)_a \in  C_{p,q}(X, \cU)$ with $\sigma \in N_{\cV}^p$ and $a \in X^q$.
\begin{lem}
\label{lem:refinement-double-complexes}
$\rho^{\cU,\cV}_{*,*}$ is a morphism of double complexes. Thus, it induces a
morphism of spectral sequences
$$
\rho^{\cU,\cV}_{p,q}: E^*_{p,q}(X,\cV)\rightarrow E^*_{p,q}(X,\cU)
$$
dependent on the choice of $\rho^{\cU,\cV}$.
\end{lem}
\begin{proof}
Let $\delta_\cV$ and $\delta_\cU$ denote the respective \cech~differentials from
$\check{\cC}_p(\cV;C^\cell_q)$ and $\check{\cC}_p(\cU;C^\cell_q)$. As we have the refinement chain
morphism $\rho^{\cU,\cV}_*:C^\cell_*(N_\cU) \to C^\cell_*(N_\cV)$ we also have commutativity
$\rho^{\cU,\cV}_{*,*} \circ \delta^\cV = \delta^{\,\cU} \circ \rho^{\cU,\cV}_{*,*}$.
This implies that $\rho^{\cU,\cV}_{*,*}$ commutes with the horizontal differential $d^H$.
For commutativity with $d^V$, we consider a generating chain $(\sigma)_a \in E^0_{p,q}(X,\cV)$
with $\sigma \in N_\cU^p$ and $a \in X^q$.
Then if $\dim(\rho^{\cU,\cV}\sigma)=p$ we have
\begin{multline}
\nonumber
\rho^{\cU,\cV}_{p,q-1} \circ d^V \big( (\sigma)_a \big) =
\rho^{\cU,\cV}_{p,q-1} \Big((-1)^p\sum_{b \preceq \overline{a}} ([b : a]\sigma)_{b} \Big) =
(-1)^p\sum_{b \preceq \overline{a}} ([b : a] \rho^{\cU,\cV} \sigma)_{b}
\\ =(-1)^p d_q^\cell \big((\rho^{\cU,\cV} \sigma)_a\big) =
d^V \circ \rho^{\cU,\cV}_{p,q}\big( (\sigma)_a \big)
\end{multline}
and for $\dim(\rho^{\cU,\cV}\sigma)<p$ commutativity follows since both terms vanish.

A morphism of double complexes gives rise to a morphism of the
vertical filtrations. By \citep[Thm.~3.5]{McCleary2001} this induces a
morphism of spectral sequences $\rho^{\cU,\cV}_{*,*}$.
\end{proof}

Since $\rho^{\cU,\cV}:N_\cU \rightarrow N_\cV$ is not unique,
the induced morphism $\rho^{\cU,\cV}_{*,*}$  on
the $0$-page will in general not be unique either. We have, however, the following:
\begin{prop}
\label{prop:refinement-independent}
The $2$-morphism obtained by restricting $\rho^{\cU,\cV}_{*,*}$ is independent
of the particular choice of refinement map $\rho^{\cU,\cV}: N_\cV \rightarrow N_\cU$.
\end{prop}
\begin{proof}
We have to show that $\rho^{\cU,\cV}_{*,*}$ is independent of the particular choice of the refinement morphism.
First, define a carrier $R:N_\cV\rightrightarrows N_\cU$
by the assignment
$$
\sigma \mapsto R(\sigma)=\big\{ \nu \in N_\cU \mid V_\sigma \subseteq U_\nu \big\}\ .
$$
The geometric realisation of the subcomplex $R(\sigma)$ is homeomorphic to a standard simplex, in particular contractible, so $R$ is acyclic. Note that $\rho^{\cU,\cV}_{*,*}$ is carried by $R$.
Hence, by Thm.~\ref{thm:acyclic-carrier} for any pair
of refinement maps $\rho^{\cU,\cV}, \tau^{\cU,\cV}\colon N_\cV \rightarrow N_\cU$,
there exists a chain homotopy $k_*:C_{n}(N_\cV)\rightarrow C_{n+1}(N_\cU)$ carried by $R$, so that
\[
k_*\delta^\cV +\delta^{\,\cU} k_*  = \tau_*^{\cU,\cV} - \rho_*^{\cU,\cV}
\]
for all $n\geq 0$ and where
$\tau_*^{\cU,\cV}$ and $\rho_*^{\cU,\cV}$ are induced morphisms of chain complexes $C_*(N_\cV)\rightarrow C_*(N_\cU)$.
In particular, using the same notation, this translates into chain homotopies
 $k_*: E^0_{p,q}(X,\cV)\rightarrow E^0_{p+1,q}(X,\cU)$
on the $0$-page such that
\[
k_*\delta^\cV + \delta^{\,\cU} k_* = \tau^{\cU,\cV}_{*,*}-\rho^{\cU,\cV}_{*,*}
\]
Thus, $\tau^{\cU,\cV}_{*,*}=\rho^{\cU,\cV}_{*,*}$ from the second page onward.
\end{proof}

\begin{figure}
  \begin{tikzpicture}
    \foreach \x in {0,1,2,3}
    {
      \pgfmathtruncatemacro{\shif}{\x * 4}
      \def\shift{\shif cm}
      \begin{scope}[xshift=\shift]
        \node [below] at (0,0) {$a$};
        \node [below] at (1,0) {$b$};
        \node [below] at (2,0) {$c$};
        \node [below] at (3,0) {$d$};
        \node [above] at (0,1) {$e$};
        \node [above] at (1,1) {$f$};
        \node [above] at (2,1) {$g$};
        \node [above] at (3,1) {$h$};
        % draw faces
        \ifnum \x > 0
          \fill [color=gray] (0,0) rectangle (1,1);
          \ifnum \x > 1
            \fill [color=gray] (1,0) rectangle (2,1);
            \ifnum \x > 2
              \fill [color=gray] (2,0) rectangle (3,1);
            \fi
          \fi
        \fi
        % draw edges
        \draw[line width=1.5] (0,0)--(3,0)--(3,1)--(0,1)--cycle;
        \ifnum \x > 0
          \draw[line width=1.5] (1,0)--(1,1);
          \ifnum \x > 1
            \draw[line width=1.5] (2,0)--(2,1);
          \fi
        \fi
        % draw bullet points
        \fill (0,0) circle (2pt);
        \fill (1,0) circle (2pt);
        \fill (2,0) circle (2pt);
        \fill (3,0) circle (2pt);
        \fill (0,1) circle (2pt);
        \fill (1,1) circle (2pt);
        \fill (2,1) circle (2pt);
        \fill (3,1) circle (2pt);
      \end{scope}
    }
  \end{tikzpicture}
  \caption{Cubical complex $\cC_*$ at values $0$,$1$,$2$ and $3$.}
  \label{fig:filtered-cubical}
\end{figure}
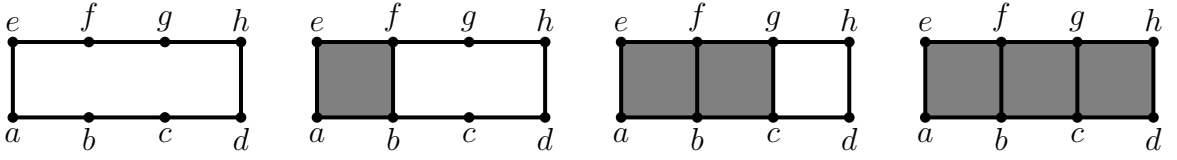
\begin{figure}
\begin{tikzpicture}
\foreach \x in {0,1,2}
{
  \pgfmathtruncatemacro{\xshif}{\x * 5}
  \def\xshift{\xshif cm};
  \foreach \y in {0,1,2}
  {
    \pgfmathtruncatemacro{\yshif}{-\y * 2}
    \def\yshift{\yshif cm};
    \begin{scope}[xshift=\xshift, yshift=\yshift]
      % draw arrows
      \if \x 0
        \draw[->, line width=2] (3.3,0.5)--(4.7,0.5);
        \if \y 1
          \node at (4,0.8) {$0$};
        \else
          \node at (4,0.8) {$\Id$};
        \fi
      \fi
      \if \x 1
        \draw[->, line width=2] (3.3,0.5)--(4.7,0.5);
        \if \y 0
          \node at (4,0.8) {$0$};
        \else
          \node at (4,0.8) {$\Id$};
        \fi
      \fi
      % draw labels
      \if \x 0
        \node at (-1,0.5) {$\y$};
      \fi
      \if \y 0
        \node at (1.5,1.5) {$\cU_\x$};
      \fi
      % draw faces
      \ifnum \y > 0
        \fill [color=gray] (0,0) rectangle (1,1);
        \ifnum \y > 1
          \fill [color=gray] (1,0) rectangle (2,1);
        \fi
      \fi
      % draw edges
      \draw[line width=1.5] (0,0)--(3,0)--(3,1)--(0,1)--cycle;
      \ifnum \y > 0
        \draw[line width=1.5] (1,0)--(1,1);
        \ifnum \y > 1
          \draw[line width=1.5] (2,0)--(2,1);
        \fi
      \fi
      % draw loop representatives
      \if \x 0
        \if \y 2
          \draw[line width=1.5, color=red] (2,0)--(3,0)--(3,1)--(2,1)--cycle;
        \else
          \if \y 0
            \fill[color=blue] (1,0) circle (5pt);
            \fill[color=blue] (1,1) circle (5pt);
          \fi
          \fill[color=blue] (2,0) circle (5pt);
          \fill[color=blue] (2,1) circle (5pt);
        \fi
      \fi
      \if \x 1
        \if \y 0
          \fill[color=blue] (1,0) circle (5pt);
          \fill[color=blue] (1,1) circle (5pt);
        \else
          \draw[line width=1.5, color=red] (1,0)--(3,0)--(3,1)--(1,1)--cycle;
        \fi
      \fi
      \if \x 2
        \draw[line width=1.5, color=red] (0,0)--(3,0)--(3,1)--(0,1)--cycle;
      \fi
      % draw bullet points
      \fill (0,0) circle (2pt);
      \fill (1,0) circle (2pt);
      \fill (2,0) circle (2pt);
      \fill (3,0) circle (2pt);
      \fill (0,1) circle (2pt);
      \fill (1,1) circle (2pt);
      \fill (2,1) circle (2pt);
      \fill (3,1) circle (2pt);
    \end{scope}
  }
}
\end{tikzpicture}
\caption{Cubical complex $\cC_*$ with covers $\cU_0$, $\cU_1$ and $\cU_2$, and
with filtration values $0$,$1$ and $2$. Blue dots represent classes in
$E_{1,0}^2(\cC,\cU_i)$ and red loops represent classes on
$E_{0,1}^2(\cC,\cU_i)$, for $i=0,1,2$.}
\label{fig:filtered-covers-cubical}
\end{figure}
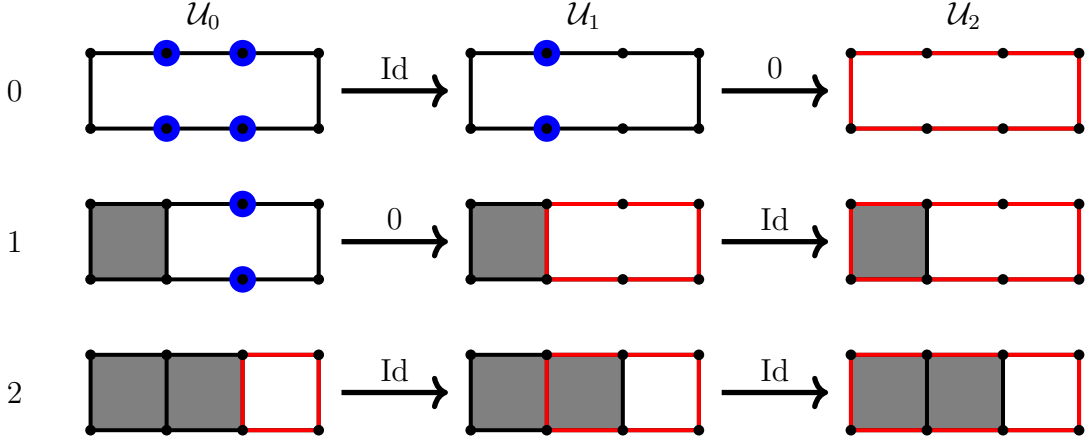

\begin{exmp}
\label{ex:cubical-filtered}
Consider a filtered cubical complex $\cC_*$. At value $0$, $\cC_*$ is given by the
vertices on $\cR^2$ at the coordinates $a=(0,0), b=(1,0), c=(2,0), d=(3,0),
e=(0,1), f=(1,1), g=(2,1), h=(3,1)$, together with all edges contained in the boundary of the rectangle $adhe$. Then, at value $1$ there appears the edge $bf$ with the
face $abfe$. At value $2$ the edge $gc$ with the face $fgcb$, and finally at
value $3$ the face $ghdc$ appears. This is depicted on
figure~\ref{fig:filtered-cubical}. Then, consider the cover $\cU_0$ by three
subcomplexes on the squares $A=(a,b,f,e)$, $B=(b,c,g,f)$ and $C=(c,d,h,g)$.
Also, we consider
the cover $\cU_1$ given by $A$ and $C \cup B$, and  $\cU_2$ given by all $\cC_*$. The induced morphisms on second page terms
at different filtratioin values are either null or the identity, as illustrated
on figure~\ref{fig:filtered-covers-cubical}.
\end{exmp}

A consequence of Prop.~\ref{prop:refinement-independent}
is that if we have a space $X$ together
with covers $\cU \prec \cV \prec \cU$, then by uniqueness
the morphism on the second page induced by the
consecutive inclusions coincides with the identity.
This gives rise to the next result.

\begin{prop}
\label{prop:refinement-iso}
Suppose a pair of covers $\cU$ and $\cV$ of $X$ are
a refinement of one another. Then there is a $2$-spectral sequence isomorphism $E^2_{*,*}(X,\cU)\simeq E^2(X,\cV)$.
\end{prop}

This corollary implies that for any cover $\cU$ of $X$,
the cover $\cU\cup X$ obtained by adding the extra
covering element $X$ is such that
the second page $E^2_{p,q}(X, \cU\cup X)$ has
only the first column nonzero.

\begin{lem}
\label{lem:cover-containing-space}
Consider a cover $\cU$ of a space $X$, and suppose that
$X \in \cU$. Then $E^2_{p,q}(X,\cU) = 0$
for all $p>0$.
\end{lem}
\begin{proof}
This follows from the observation that the cover $\{X\}$ consisting of a single
element satisfies $\{X\} \prec \cU \prec \{X\}$.
Using Prop.~\ref{prop:refinement-iso} we therefore obtain isomorphisms
$E^2_{p,q}(X,\cU)\simeq E^2_{p,q}(X,\{X\})$, and
the result follows.
\end{proof}

Suppose that none of the two covers $\cV$ and $\cU$ refines the other.
One can still compare them using the
common refinement $\cV\cap\cU=\big\{V\cap U\big\}_{V \in \cV, U \in \cU}$
which is a cover of $X$. Thus, there are two refinement morphisms
\begin{equation}
\label{eq:double-interesect-refinements}
\begin{tikzcd}
E^2_{p,q}(X, \cU) &
E^2_{p,q}(X, \cV\cap\cU)
\ar[l,"\rho^{\cU, \cV\cap\cU}_{p,q}",swap]
\ar[r,"\rho^{\cV, \cV\cap\cU}_{p,q}"]  &
E^2_{p,q}(X, \cV).
\end{tikzcd}
\end{equation}
Following \citep[Sec.~28]{Serre1955} we can now build the double complex
$C_{p,q}(\cV, \cU, \PH_k)$
which is for each $k\geq 0$ given by
\[
\begin{tikzcd}[column sep=4cm, row sep=1.1cm]
\bigoplus \limits_{\substack{\sigma \in N^{p+1}_\cV \\ \tau \in N^{q}_\cU}}
\PH_k(\cV_\sigma\cap \cU_\tau) \ar[d,"\delta^\cV"] &
\ar[l,"(-1)^{p+1}\delta^{\,\cU}" above] \ar[d,"\delta^\cV"]
\bigoplus \limits_{\substack{\sigma \in N^{p+1}_\cV \\ \tau \in N^{q+1}_\cU}} \PH_k(\cV_\sigma\cap \cU_\tau) \\
\bigoplus \limits_{\substack{\sigma \in N^{p}_\cV \\ \tau \in N^{q}_\cU}} \PH_k(\cV_\sigma\cap \cU_\tau) &
\ar[l,"(-1)^{p}\delta^{\,\cU}" above]
\bigoplus \limits_{\substack{\sigma \in N^{p}_\cV \\ \tau \in N^{q+1}_\cU}} \PH_k(\cV_\sigma\cap \cU_\tau)
\end{tikzcd}
\]
for any pair of integers $p,q\geq 0$.
From this double complex we can study the two associated spectral sequences
\begin{align*}
{^{\rm I}E}_{p,q}^1(\cV,\cU;\PH_k) &= \bigoplus \limits_{\sigma \in N^p_\cV}
\check{\cH}_q\big((-1)^p\delta^{\,\cU}|_{\cV_\sigma\cap\cU};\PH_k\big) \ ,\\
{^{\rm II}E_{p,q}^1(\cV,\cU;\PH_k)} &= \bigoplus \limits_{\tau \in N^q_\cU}
\check{\cH}_p\big(\,\delta^\cV|_{\cV\cap\cU_\tau};\PH_k\big)\ ,
\end{align*}
whose common target of convergence is
$
\check{\cH}_{n}(\cV\cap \cU;\PH_k)
$
with $p+q = n$. For details about the spectral sequence associated to
a double complex, the reader is recommended to look at \citep[Thm.~2.15]{McCleary2001}.

\begin{exmp}
Consider the cubical complex $\cC_*$ from example~\ref{ex:cubical-filtered}.
Set $\cU=\cU_1$, that is, $\cU$ is the cover by the sets $U_1 = A$ and
$U_2 = A \cup B$. On the other hand, consider $\cV$ to be formed of
$V_1 = A \cup B$ and $V_2 = C$. The double complex
$C_{p,q}(\cC, \cV, \cU, \PH_k)$ is illustrated on
figure~\ref{diag:double-cover-complex} for filtration values $0$ and $1$, and
for $k = 0$. One can see that the refinement morphisms
from~(\ref{eq:double-interesect-refinements})
are actually projections.
\begin{figure}
\begin{tikzpicture}[scale=0.7]
  \foreach \x in {0,1}
  {
  \tiny
  \pgfmathtruncatemacro{\xshift}{12*\x + 1}
  \begin{scope}[xshift=\xshift cm]
    % fill cover
    \fill[color=green, semitransparent] (-0.3,-0.3) rectangle (0.3,1.3);
    % cover label
    \node at (-2, 0.5) {$\big( V_1\cap V_2 \big) \cap U_1$};
    % draw arrows
    \draw[->, line width=2] (0,-0.4)--(0,-1.6);
    % draw representative
    \fill[color=blue] (0,0) circle (5pt);
    \fill[color=blue] (0,1) circle (5pt);
  \end{scope}
  \pgfmathtruncatemacro{\xshift}{12*\x - 2}
  \begin{scope}[xshift=\xshift cm, yshift=-3cm]
    % fill cover
    \fill[color=orange, semitransparent] (-0.3,-0.3) rectangle (1.3,1.3);
    % cover label
    \node at (0.5,-0.8) {$V_1 \cap U_1$};
    % fill face
    \if \x 1:
      \fill[color=gray] (0,0)--(1,0)--(1,1)--(0,1)--cycle;
    \fi
    % draw edges
    \draw[line width=1.5]
      (1,0)--(0,0)--(0,1)--(1,1);
    % draw bullet points
    \fill (0,0) circle (2pt);
    \fill (0,1) circle (2pt);
    % draw representative or edge killing them
    \if \x 1:
      \draw[line width=1.5] (1,0)--(1,1);
      \fill (1,0) circle (2pt);
      \fill (1,1) circle (2pt);
    \else
      \fill[color=red] (1,0) circle (5pt);
      \fill[color=red] (1,1) circle (5pt);
    \fi
  \end{scope}
  \pgfmathtruncatemacro{\xshift}{12*\x}
  \begin{scope}[xshift= \xshift cm, yshift=-3cm]
    % fill cover
    \fill[color=orange, semitransparent] (-0.3,-0.3) rectangle (1.3,1.3);
    % cover label
    \node at (0.5,-0.8) {$V_1 \cap U_2$};
    % draw edges
    \draw[line width=1.5] (0,0)--(1,0);
    \draw[line width=1.5] (0,1)--(1,1);
    % draw bullet points
    \fill (1,0) circle (2pt);
    \fill (1,1) circle (2pt);
    % draw representative
    \fill[color=blue] (1,0) circle (5pt);
    \fill[color=blue] (1,1) circle (5pt);
    % draw representative or edge killing them
    \if \x 1:
      \draw[line width=1.5] (0,0)--(0,1);
      \fill (0,0) circle (2pt);
      \fill (0,1) circle (2pt);
    \else
      \fill[color=red] (0,0) circle (5pt);
      \fill[color=red] (0,1) circle (5pt);
    \fi
  \end{scope}
  \pgfmathtruncatemacro{\xshift}{12*\x + 2}
  \begin{scope}[xshift=\xshift cm, yshift=-3cm]
    % fill cover
    \fill[color=orange, semitransparent] (-0.3,-0.3) rectangle (1.3,1.3);
    % cover label
    \node at (0.5,-0.8) {$V_2 \cap U_2$};
    % draw edges
    \draw[line width=1.5] (0,0)--(1,0)--(1,1)--(0,1);
    % draw bullet points
    \fill (0,0) circle (2pt);
    \fill (1,0) circle (2pt);
    \fill (0,1) circle (2pt);
    \fill (1,1) circle (2pt);
    % draw representatives
    \fill[color=blue] (0,0) circle (5pt);
    \fill[color=blue] (0,1) circle (5pt);
  \end{scope}
  \pgfmathtruncatemacro{\xshift}{12*\x + 6}
  \begin{scope}[xshift=\xshift cm, yshift=-3cm]
    % fill cover
    \fill[color=green, semitransparent] (-0.3,-0.3) rectangle (0.3,1.3);
    \node at (0,-0.8) {$V_1 \cap (U_1 \cap U_2)$};
    % draw arrows
    \draw[->, line width=2] (-0.6,0.5)--(-2,0.5);
    % draw representative or edge killing them
    \if \x 1:
      \draw[line width=1.5] (0,0)--(0,1);
      \fill (0,0) circle (2pt);
      \fill (0,1) circle (2pt);
    \else
      \fill[color=red] (0,0) circle (5pt);
      \fill[color=red] (0,1) circle (5pt);
    \fi
  \end{scope}
  }
\end{tikzpicture}
\caption{
$C_{p,q}(\cC, \cV, \cU, \PH_k)$ at filtration values $0$
and $1$.}
\label{diag:double-cover-complex}
\end{figure}
\end{exmp}

Consider the nerve $N_{\cU\cap \cV}$ as a subset of the product of nerves $N_\cU \times N_\cV$.
We have then two projections $\pi^\cU:N_{\cU\cap \cV} \to N_\cU$ and $\pi^\cV:N_{\cU\cap \cV} \to N_\cV$,
both of which induce chain morphisms $\pi^\cU_*:C_*(N_{\cU\cap \cV}) \to C_*(N_\cU)$ and
$\pi^\cV_*:C_*(N_{\cU\cap \cV}) \to C_*(N_\cV)$. These induce a pair of morphisms
$$
\begin{tikzcd}
\bigoplus \limits_{\sigma \in N^p_\cU}C_k^\cell(\cV_\sigma) &  \bigoplus \limits_{\substack{\sigma \in N^p_\cU\\ \tau \in N^q_\cV}}
C_k^\cell(\cV_\sigma \cap \cU_\tau) \ar[r, "\pi^\cU_{q,k}"] \ar[l, swap, "\pi^\cV_{p,k}"]  &
\bigoplus \limits_{\tau \in N^q_\cV} C_k^\cell(\cU_\tau)\ ,
\end{tikzcd}
$$
for any pair of integers $p,q\geq 0$.
The induced map $\pi^\cV_{p,k}$ on $C_k(\cV_\sigma \cap \cU_\tau)$ satisfies
\[
\pi^\cV_{p,k}\left( \left(\sigma \times \tau\right)_s \right) = \begin{cases}
(\sigma)_s & \text{if } \dim(\tau) = 0 ,\\
0 & \text{else,}
\end{cases}
\]
for all $\sigma \in N_{\cV}^p, \tau \in N_{\cU}$ and a cell $a \in (\cV_\sigma \cap \cU_\tau)^k$. The map $\pi^\cU_{*,*}$ acts similarly.
By definition $\pi^\cU_{*,*}$ and $\pi^\cV_{*,*}$ both commute with the \cech~differentials
$\delta^{\cU}$ and $\delta^\cV$ respectively. Let $\sigma \in N_{\cV}^p$ and $\tau \in N_{\cU}^0$. Then we have
\[
\begin{tikzcd}
\left( \sigma \times \tau \right)_a \ar[r, "\pi^\cV_{*,*}"] \ar[d, "d_n"] &
\left(\sigma \right)_a \ar[d, "d_n"] \\
\sum_{b \in \overline{a}}\big( \mbox{[$b$ : $a$]} \sigma \times \tau \big)_{b} \ar[r, "\pi^\cV_{*,*}"] &
\sum_{b \in \overline{a}}\left( \mbox{[$b$ : $a$]} \sigma \right)_{b}
\end{tikzcd}
\]
for all cells $a \in (\cV_\sigma \cap \cU_\tau)^k$.
This implies that $\pi^\cV_{*,*}$ commutes with $d_n$ and the same holds for~$\pi^\cU_{*,*}$.
We obtain a morphism
\[
\pi^\cV_{p,k} \colon \bigoplus_{\substack{\sigma \in N^p_\cV \\ \tau \in N^0_\cU}} C_k(\cV_\sigma \cap \cU_\tau) \to \bigoplus_{\sigma \in N^p_\cV } C_k(\cV_\sigma)\ ,
\]
commuting with $d_*$ and $\delta^\cV$ and $\delta^\cU$. Thinking of the $0$-th column ${^{\rm I}E}_{p,0}^0(\cV, \cU;\PH_k)$ as a chain complex with \cech~differential $\delta^\cV$, one has a chain morphism
$$
\pi^{\cV}_{p,k}:{^{\rm I}E}_{p,0}^1(\cV, \cU;\PH_k)
\rightarrow \check{\cC}_p(\cV; \PH_k) = E^1_{p,k}(X,\cV)
$$
for all $p\geq 0$.
By the same argument there is another chain morphism
$$
\pi^{\cU}_{q,k}:{^{\rm II}E}_{0,q}^1(\cV, \cU;\PH_k)
\rightarrow \check{\cC}_q(\cU; \PH_k)= E^1_{q,k}(X,\cU)\ .
$$
for all $q\geq 0$.
There is a very natural way of understanding how much $\pi^\cV_{p,k}$ fails to be an isomorphism.
To do this, we take for each simplex $\sigma \in N_\cV^p$, the
Mayer-Vietoris spectral sequence for $\cV_\sigma$ covered by $\cV_\sigma\cap\cU$
$$
M^2_{q,k}(\cV_\sigma \cap \cU) \Rightarrow
\PH_{q+k}(\cV_\sigma),
$$
where we changed the notation from $E^2_{q,k}(\cV_\sigma, \cV_\sigma \cap \cU)$ to $M^2_{q,k}(\cV_\sigma \cap \cU)$ as it helps distinguishing this spectral sequence from $^{I}E^*_{p,q}$.
Then
$${^{\rm I}E}_{p,0}^1(\cV, \cU;\PH_k) = \bigoplus\limits_{\sigma \in N_\cV^p}
M^2_{0,k}(\cV_\sigma \cap \cU)\ .
$$
Here we notice that the restriction of $\pi^\cV_{p,k}$ to the summand $M^2_{0,k}(\cV_\sigma\cap\cU)$ is
given by the composition
$$
\begin{tikzcd}
M^2_{0,k}(\cV_\sigma\cap \cU) \ar[r,twoheadrightarrow] &
M^\infty_{0,k}(\cV_\sigma \cap \cU)
\ar[r,hookrightarrow] & \PH_k(\cV_\sigma).
\end{tikzcd}
$$
In consequence there is an induced morphism on the second page $0^{\rm th}$ column
$\pi^{\cV}_{p,k}:{^{\rm I}E}^2_{p,0}(\cV, \cU;\PH_k) \rightarrow
\check{\cH}_p(\cV;\PH_k)$.
Notice that $\PH_0$ is a cosheaf, and in this case
$M^2_{0,0}(\cV_\sigma \cap \cU) = \PH_0(\cV_\sigma)$
for all $\sigma \in N_\cV^p$.
This implies that $\pi^{\cV}_{p,0}$ is an isomorphism for all $p\geq 0$.

Now we turn to the morphism $\theta^{\cV, \cV\cap\cU}_{p,k}$ defined
by the composition
$$
\begin{tikzcd}
\check{\cH}_p(\cV\cap\cU;\PH_k) \ar[twoheadrightarrow]{r} &
^{\rm I}E^\infty_{p,0}(\cV, \cU;\PH_k) \ar[hookrightarrow]{r} &
^{\rm I}E^2_{p,0}(\cV, \cU, \PH_k) \ar["\pi^{\cV}_{p,k}"]{r} &
\check{\cH}_p(\cV; \PH_k)
\end{tikzcd}
$$
and notice that $\theta^{\cV, \cV\cap\cU}_{p,k}$ is carried by the acyclic
carrier sending $\sigma \times \tau \in N_{\cV \cap \cU}$ to $\Delta^\sigma \subseteq N_\cV$.
In particular, if $\cV \prec \cU$ then
by using Lemma~\ref{lem:cover-containing-space} we have
${^{\rm I}E}^1_{p,q} = 0$ for all $q>0$ and the
first two arrows in the definition of $\theta^{\cV, \cV\cap\cU}_{p,k}$
are isomorphisms. Similarly, in this case
we obtain $M^2_{q,k} = 0$ for all $q>0$, and
$\pi^{\cV}_{p,k}$ becomes an isomorphism.
Altogether, the inverse
$(\theta^{\cV, \cV\cap\cU}_{p,k})^{-1}$ is well-defined,
and by composition we define
morphisms $\theta^{\cU, \cV}_{p,k} = \theta^{\cU, \cV\cap\cU}_{p,k}
\circ (\theta^{\cV, \cV\cap\cU}_{p,k})^{-1}$.
Here notice that $\theta^{\cU, \cV\cap\cU}_{p,k}$ is defined in an analogous way
to $\theta^{\cV, \cV\cap\cU}_{p,k}$, but using $^{\rm II}E^*_{0,p}(\cV, \cU;\PH_k)$
instead of $^{\rm I}E^*_{p,0}(\cV, \cU;\PH_k)$.
The following proposition should also follow
from applying an appropriate version of the universal
coefficient theorem to~\citep[Prop.~4.4]{Serre1955}.
Instead we will prove the dual statement of this proposition by means of acyclic carriers.

\begin{prop}
\label{prop:equal-refinement}
Suppose that  $\cV \prec \cU$, and let
$\rho^{\cU, \cV}$ denote a refinement map.
The morphism $\theta^{\cU, \cV}_{p,k}:
E_{p,k}^2(X,\cV) \rightarrow E_{p,k}^2(X,\cU)$
coincides with the standard morphism induced by $\rho^{\cU, \cV}$.
\end{prop}
\begin{proof}
Since $\cV \prec \cU$, the morphism $\theta^{\cV, \cV\cap\cU}_{p,k}:
\check{\cH}_p(\cV\cap\cU, \PH_k) \rightarrow \check{\cH}_p(\cV,\PH_k)$ is
an isomorphism. Now consider the diagram
$$
\begin{tikzcd}
\check{\cH}_p(\cV;\PH_k) \ar[rr,"\rho^{\cU, \cV}_{p,k}"]
\ar[d,"\simeq", leftarrow] & &
\check{\cH}_p (\cU;\PH_k) \\
\check{\cH}_p(\cV\cap \cU;\PH_k) \ar[r,twoheadrightarrow] &
^{\rm II} E^\infty_{0,p}(\cV,\cU;\PH_k) \ar[r,hookrightarrow] &
^{\rm II} E^2_{0,p}(\cV,\cU;\PH_k). \ar[u, "\pi^{\cU}_{p,k}", swap]
\end{tikzcd}
$$
To check that it commutes we study the following triangle of acyclic carriers
$$
\begin{tikzcd}
& N_{\cV\cap\cU} \ar[rd, "\pi^\cU"]
&  \\
N_{\cV} \ar[ru, shift left=0.5ex, "F"] \ar[ru, shift right=0.5ex]
  \ar[rr, shift left=0.5ex, "R"] \ar[rr, shift right=0.5ex] & &
N_{\cU}
\end{tikzcd}
$$
where $R$ is defined in Prop.~\ref{prop:refinement-independent}. The
carrier $F$ is given for every $\sigma \in N_\cV$ by
$F(\sigma) = \Delta^\sigma \times |R(\sigma)|$. Since $F$ is acyclic, there exists  $f_*:C_*(N_\cV)\rightarrow C_*(N_{\cV\cap\cU})$ inducing a chain morphism $f_*:\check{\cC}_p(\cV, S_k)\rightarrow \check{\cC}_p(\cV\cap\cU, S_k)$ by the assignment $(\sigma)_s\mapsto (f_*(\sigma))_s$ for all simplices $s \in S_k(X)$ and all $\sigma \in N_\cV$.
In fact, $F$ defines an acyclic equivalence by considering the inverse carrier $P:N_{\cV\cap\cU}\rightrightarrows N_\cV$ sending $\sigma \times \tau$ to $\Delta^\sigma$. In this case the shift carrier
$I_\cV: N_\cV \rightrightarrows N_\cV$ is given by the assignment
$\sigma \mapsto \Delta^\sigma$,
and $I_{\cV\cap\cU}: N_{\cV\cap\cU} \rightrightarrows N_{\cV\cap\cU}$ is
given by $\sigma \times \tau \mapsto \Delta^\sigma \times |R(\sigma)|$.
As $\theta_{p,k}^{\cV,\cV\cap\cU}$ is carried by $P$, this implies that $f_*=\big(\theta_{p,k}^{\cV,\cV\cap\cU}\big)^{-1}$ as morphisms $\check{\cH}_p(\cV,\PH_k)\rightarrow \check{\cH}_p(\cV\cap\cU,\PH_k)$. Consequently, $\theta^{\cU,\cV}_{p,k}$ is carried by $\pi^\cU F = R$. Altogether, we obtain the equality $\theta_{p,k}^{\cU,\cV}=\rho_{p,k}^{\cU,\cV}$ as  morphisms $\check{\cH}_p(\cV,\PH_k)\to \check{\cH}_p(\cU,\PH_k)$ since both are carried by $R$.
\end{proof}

Still assuming that $\cV\prec \cU$, we now look for
conditions for the existence of an inverse
$\varphi_{p,k}^{\cV, \cU}:E_{p,k}^2(X,\cU) \rightarrow E_{p,k}^2(X,\cV)
$ of $\theta^{\cU, \cV}_{p,k}$.
\begin{prop}
\label{prop:covers-strong}
Suppose that $\cV \prec \cU$. If
$M^2_{p,k}(\cV \cap \cU_\tau) = 0$ for all $p > 0$, $k\geq 0$
and all $\tau \in N_\cU^q$, then the maps $\theta_{*,*}^{\cU, \cV}$ induce a 2-isomorphism of spectral sequences
\[
    E_{*,*}^{\geq 2}(X,\cU) \simeq E^{\geq 2}_{*,*}(X,\cV).
\]
\end{prop}
\begin{proof}
By Prop.~\ref{prop:refinement-independent} and~Prop.\ref{prop:equal-refinement}
we can choose a refinement map $\rho^{\cU, \cV}:N_{\cV}\rightarrow N_\cU$ giving
a morphism of spectral sequences
\[
    \rho^{\cU, \cV}_{*,*}:E^{\geq 2}_{*,*}(X,\cV) \rightarrow E^{\geq 2}_{*,*}(X,\cU)
\]
that coincides with $\theta_{*,*}^{\cU, \cV}$. Our assumption about $M^2_{p,k}$ implies
${^{\rm II}E}_{p,q}^2(\cV,\cU;\PH_k) = 0$ for all $p>0$, which
in turn gives
\begin{equation}
\label{eq:ker-trivial}
\Ker\Big(\check{\cH}_q(\cV\cap\cU;\PH_k) \twoheadrightarrow
\, {^{\rm II}E}^\infty_{0,q}(\cV, \cU;\PH_k)\Big) = 0
\end{equation}
and
\begin{equation}
\label{eq:coker-trivial}
\Coker\Big({^{\rm II}E}^\infty_{0,q}(\cV, \cU;\PH_k) \hookrightarrow
{^{\rm II}E}^2_{0,q}(\cV, \cU, \PH_k) \Big) = 0.
\end{equation}
Now note that $\pi_{q,k}^{\cU}$ yields an isomorphism ${^{\rm II}E}^2_{0,q}(\cV,\cU,\PH_k)
\simeq \check{\cH}_q (\cU,\PH_k)$. This shows that $\theta_{q,k}^{\cU, \cV}$ is a composition of isomorphisms; thus the statement follows.
\end{proof}

We will now relax the conditions in Prop.~\ref{prop:covers-strong} and use the
relations of \emph{left-interleaving} and \emph{right-interleaving}
of persistence modules (denoted by $\sim^\veps_L$ and $\sim^\veps_R$, respectively) to achieve this~(see \citep[Sec.~4]{GovSkr2018}). We have to adapt~\citep[Prop.~4.14]{GovSkr2018}.
\begin{lem}
\label{lem:LR-interleavings}
Suppose that we have persistence modules $A$, $B$ and $C$, and
a parameter $\veps\geq 0$ such that
$A\sim^\veps_R B$ and $B \sim^\veps_L C$. Denote by $\Phi$ the
morphism $\Phi:A \rightarrow C$ given by the composition
$A \twoheadrightarrow B \hookrightarrow C$.
Then there exists $\Psi:C \rightarrow A[2\veps]$
such that $\Phi$ and $\Psi$ define a $2\veps$-interleaving $A\sim^{2\veps} C$.
\end{lem}
\begin{proof}
By hypothesis, we have a sequence
$$
\begin{tikzcd}
\cE_1 \ar[r] &
A \ar[r,twoheadrightarrow, "f"] &
B \ar[r,hookrightarrow, "g"] &
C \ar[r] &
\cE_2
\end{tikzcd}
$$
which is exact in $A$ and $C$ and
where $\cE_1 \sim^\veps 0$ and $\cE_2 \sim^\veps 0$.
Then, let $v \in C$ and notice that $\Sigma^\veps C(v) \in \Img(g)$.
Thus, there exists a unique vector $w \in B$ such that $g(w) = \Sigma^\veps C (v)$.
On the other hand, there exists $z \in A$, not necessarily unique, such that
$f(z) = w$. This defines a unique element $\Sigma^\veps A(z) \in A$. To see this,
suppose that another $z' \in A$ is such that $f(z')=w$. Then $f(z-z') = 0$ and
$z-z' \in \Ker(f)$, which implies
$0 = \Sigma^\veps A(z-z') = \Sigma^\veps A(z) - \Sigma^\veps A(z')$,
and then $\Sigma^\veps A(z) = \Sigma^\veps A(z')$.
Altogether, we set $\Psi =\Sigma^\veps A \circ \Phi^{-1} \circ \Sigma^\veps C$,
which is well-defined.
\end{proof}
Recall that for $\cV \prec \cU$ we have that
$\check{\cH}_q(\cV;\PH_k)\simeq \check{\cH}_q(\cV\cap\cU;\PH_k)$
for all $k \geq 0$ and $q \geq 0$. There is a natural way to relax \eqref{eq:ker-trivial} and \eqref{eq:coker-trivial} to the persistent case. We assume that
for $\veps \geq 0$, there are right and left interleavings
\begin{equation}
\label{eq:RL-interleavings-IIE}
\check{\cH}_q(\cV\cap\cU;\PH_k) \sim_R^\veps
{^{\rm II}E}^\infty_{0,q}(\cV, \cU;\PH_k) \sim_L^\veps
{^{\rm II}E}^2_{0,q}(\cV, \cU, \PH_k).
\end{equation}
If we define $\Phi_{q,k}:\check{\cH}_q(\cV\cap\cU;\PH_k) \rightarrow
{^{\rm II}E}^2_{0,q}(\cV, \cU, \PH_k)$ to be the composition of the associated persistence morphisms as in Lem.~\ref{lem:LR-interleavings}, then there exists
\[
    \Psi_{q,k} : {^{\rm II}E}^2_{0,q}(\cV, \cU, \PH_k) \rightarrow
\check{\cH}_q(\cV\cap\cU;\PH_k)[2\veps],
\]
such that $\Phi_{q,k}$ and $\Psi_{q,k}$ define a
$2\veps$-interleaving. We repeat this argument for the local Mayer-Vietoris
spectral sequences. Assume that for some $\nu \geq 0$ there are interleavings
\begin{equation}
\label{eq:RL-interleavings-M}
{^{\rm II}E}^1_{0,q}(\cV, \cU, \PH_k) \sim_R^\nu
\bigoplus \limits_{\tau \in N^q_\cU}
M^\infty_{k,0}(\cV\cap\cU_\tau)
\sim_L^\nu
\bigoplus \limits_{\tau \in N^q_\cU} \PH_k(\cU_\tau).
\end{equation}
Let $\Pi_{q,k}:{^{\rm II}E}^1_{0,q}(\cV, \cU, \PH_k)
\rightarrow \bigoplus_{\tau \in N^q_\cU} \PH_k(\cU_\tau)$ be
the composition of the associated morphisms. By Lem.~\ref{lem:LR-interleavings} there
exists $\Xi_{q,k}$ such that $\Pi_{q,k}$
and $\Xi_{q,k}$ define a $2\nu$-interleaving.
By slight abuse of notation we will continue to denote the induced
$2\nu$-interleaving between ${^{\rm II}E}^2_{0,q}(\cV, \cU, \PH_k)$ and
$\check{\cH}_q(\cU;\PH_*)$ by $\Pi_{q,k}$ and $\Xi_{q,k}$.
Altogether we have that
$\theta^{\cU, \cV}_{q,k} = \Pi_{q,k}\circ\Phi_{q,k}\circ (\theta^{\cV, \cV\cap\cU}_{q,k})^{-1}$ and in this situation there is indeed an `inverse'
$\varphi^{\cV, \cU}_{q,k} = \theta^{\cV, \cV\cap\cU}_{q,k} \circ \Psi_{q,k} \circ \Xi_{q,k}$, which increases the persistence values by $2(\veps+\nu)$.

\begin{thm}
\label{thm:inverse-refinement}
Suppose that $\cV\prec\cU$ and for $\veps\geq0$ and $\nu\geq 0$
the interleavings in~(\ref{eq:RL-interleavings-IIE})
and~(\ref{eq:RL-interleavings-M}) hold.
Then
$$
\varphi^{\cV, \cU}_{p,q} : E^*_{p,q}(X, \cU)
\rightarrow E^*_{p,q}(X, \cV)[2(\veps+\nu)]
$$
defines a $2$-morphism of spectral sequences such that
 $\theta^{\cU, \cV}_{p,q}$ and $\varphi^{\cV, \cU}_{p,q}$
is a $2$-page $2(\veps+\nu)$-interleaving between
$E^*_{p,q}(X, \cU)$ and $E^*_{p,q}(X, \cV)$.
\end{thm}
\begin{proof}
The only thing that remains to be proved is that $\psi^{\cV,\cU}_{p,q}$ commutes with the spectral sequence differentials $d_n$ for all $n\geq 2$. Since these differentials commute with the shift morphisms $\Sigma^{2(\veps + \nu)}$, this follows from considering the diagram
%{\small
\[
\begin{tikzcd}[every matrix/.append style={nodes={font=\small}},row sep=1.5cm,column sep=0.6cm]
E^n_{p,q}(X,\cU) \ar[rrr, "d_n"] \ar[dd, swap,"\psi^{\cV,\cU}_{p,q}"] & & &
E^n_{p-n,q+n-1}(X,\cU) \ar[dd, "\psi^{\cV,\cU}_{p-n, q+n-1}"] \\
& E^n_{p,q}(X,\cV) \ar[r,"d_n"] \ar[ul, swap, "\rho^{\cU,\cV}_{p,q}"] \ar[ld, "\Sigma^{2(\veps+\nu)}"] &
E^n_{p-n, q+n-1}(X,\cV) \ar[ru, "\rho^{\cU,\cV}_{p-n,q+n-1}"] \ar[rd,swap,  "\Sigma^{2(\veps+\nu)}"]\\
E^n_{p,q}(X,\cV)[2(\veps+\nu)]\ar[rrr, "d_n"] & & &
E^n_{p-n,q+n-1}(X,\cV)[2(\veps+\nu)]\ ,
\end{tikzcd}
\]
in which the two trapeziums and the two triangles commute.
\end{proof}

\begin{exmp}
Consider a cubical complex $\cC_*$ as shown in Fig.~\ref{fig:cubical-interleaving-filt},
together with the covers $\cV = \{ \overline{A}, \overline{B}, \overline{C}, \overline{D}\}$
and $\cU = \{ \overline{A \cup B}, \overline{C \cup D}\}$, see Fig.~\ref{fig:cubical-interleaving-filt} for the cells $A$,$B$,$C$ and $D$. In this case we have
$$
\check{\cH}_1(\cV; \PH_0)\simeq \check{\cH}_1(\cV\cap \cU;\PH_0) \simeq
{\rm I}(0,1+\veps) \oplus {\rm I}(1,1+\veps) \sim^\veps
{\rm I}(0,1) \simeq {^{\rm II}E}^2_{0,1}(\cV, \cU, \PH_0)
$$
and also
$$
{^{\rm II}E}^1_{0,0}(\cV, \cU, \PH_1) \simeq 0 \sim^\veps
{\rm I}(1,1+\veps)\oplus {\rm I}(1,1+\veps)  \simeq
\bigoplus \limits_{\dim(\tau)=0} \PH_1(\cU_\tau).
$$
These interleavings are shown in Fig.~\ref{fig:cubical-interleaving-ss}.
Thm.~\ref{thm:inverse-refinement} implies that there is a $4\veps$-interleaving between
$E^*_{p,q}(X, \cU)$ and $E^*_{p,q}(X, \cV)$. Notice that in this
example, the nontrivial interleaved terms are in different positions of the
spectral sequences. Therefore we can improve the upper bound to $2\veps$.
We will use this observation later in Prop.~\ref{prop:local-checks}.
\begin{figure}
  \begin{tikzpicture}
    \foreach \a in {0,1,2}
    {
      \pgfmathtruncatemacro{\ashif}{\a * 5}
      \def\ashift{\ashif cm};
      \begin{scope}[xshift=\ashift]
        % faces
        \if \a 2
          \fill [color=gray] (0,0) rectangle (2,2);
        \fi
        % edges
        \draw[line width=1.5] (0,0)--(2,0)--(2,2)--(0,2)--cycle;
        \ifnum \a > 0
          \draw[line width=1.5] (0,1)--(2,1);
          \ifnum \a > 1
            \draw[line width=1.5] (1,0)--(1,2);
          \fi
        \fi
        % draw vertices of cubical complex
        \foreach \x in {0,1,2}{
          \foreach \y in {0,1,2}{
            \fill (\x, \y) circle (2pt);
        }}
        % take out middle point for \a == 0
        \if \a 0
          \fill [color=white] (1,1) circle (3pt);
        \fi
        % labels
        \if \a 2
          \node[color=white] at (0.5,0.5) {$A$};
          \node[color=white] at (1.5,0.5) {$B$};
          \node[color=white] at (0.5,1.5) {$C$};
          \node[color=white] at (1.5,1.5) {$D$};
        \fi
      \end{scope}
    }
  \end{tikzpicture}
  \caption{Cubical complex $\cC_*$ at values $0$,$1$ and $1+\veps$.}
  \label{fig:cubical-interleaving-filt}
\end{figure}

\begin{figure}
  \begin{tikzpicture}[scale=0.9]
    \foreach \a in {0,1}
    {
      \pgfmathtruncatemacro{\ashif}{\a *9}
      \def\ashift{\ashif cm};
      \begin{scope}[xshift=\ashift]
        % draw lines
        \draw[line width=1.5] (1,0)--(0,0)--(0,1);
        \draw[line width=1.5] (1.5,0)--(2.5,0)--(2.5,1);
        \draw[line width=1.5] (0,1.5)--(0,2.5)--(1,2.5);
        \draw[line width=1.5] (1.5,2.5)--(2.5,2.5)--(2.5,1.5);
        \if \a 1
          \draw[line width=1.5] (0,1)--(1,1);
          \draw[line width=1.5] (0,1.5)--(1,1.5);
          \draw[line width=1.5] (1.5,1)--(2.5,1);
          \draw[line width=1.5] (1.5,1.5)--(2.5,1.5);
        \fi
        % draw vertices of cubical complex
        \foreach \x in {0,2.5}{
          \foreach \y in {0,1,1.5,2.5}{
            \fill (\x, \y) circle (2pt);
          }
        }
        \if \a 0
          \fill[color=blue] (0,1) circle (3pt);
          \fill[color=blue] (0,1.5) circle (3pt);
          \fill[color=blue] (1,0) circle (3pt);
          \fill[color=blue] (1,2.5) circle (3pt);
          \fill[color=blue] (1.5,0) circle (3pt);
          \fill[color=blue] (1.5,2.5) circle (3pt);
          \fill[color=blue] (2.5,1) circle (3pt);
          \fill[color=blue] (2.5,1.5) circle (3pt);
        \fi

        \draw[->, line width=2] (2.8, 1.25)--(3.8, 1.25);
        \if \a 0
          \node at (3.3, 1.8) {$\Id$};
        \fi
        \if \a 1
          \node at (3.3, 1.8) {$0$};
        \fi
        \if \a 1
          \fill[color=red] (1,0) circle (3pt);
          \fill[color=red] (1.5,0) circle (3pt);
          \fill[color=red] (1,1) circle (3pt);
          \fill[color=red] (1.5,1) circle (3pt);
          \fill[color=blue] (1,1.5) circle (3pt);
          \fill[color=blue] (1.5,1.5) circle (3pt);
          \fill[color=blue] (1,2.5) circle (3pt);
          \fill[color=blue] (1.5,2.5) circle (3pt);
        \fi
      \end{scope}
    }
    \foreach \a in {0,1}
    {
      \pgfmathtruncatemacro{\ashif}{4.5 + \a * 9}
      \def\ashift{\ashif cm};
      \begin{scope}[xshift=\ashift]
        % draw lines
        \if \a 0
          \draw[line width=1.5] (0,1)--(0,0)--(2,0)--(2,1);
          \draw[line width=1.5] (0,1.5)--(0,2.5)--(2,2.5)--(2,1.5);
        \fi
        \if \a 1
          \draw[line width=1.5, color=red] (0,0)--(2,0)--(2,1)--(0,1)--cycle;
          \draw[line width=1.5, color=blue] (0,1.5)--(2,1.5)--(2,2.5)--(0,2.5)--cycle;
        \fi
        % draw vertices of cubical complex
        \foreach \x in {0,1,2}{
          \foreach \y in {0,1,1.5,2.5}{
            \fill (\x, \y) circle (2pt);
          }
        }
        \if \a 0
          \fill[color=white] (1,1) circle (3pt);
          \fill[color=white] (1,1.5) circle (3pt);
          \fill[color=blue] (0,1) circle (3pt);
          \fill[color=blue] (0,1.5) circle (3pt);
          \fill[color=blue] (2,1) circle (3pt);
          \fill[color=blue] (2,1.5) circle (3pt);
        \fi
      \end{scope}
    }
  \end{tikzpicture}
  \caption{Morphisms $\theta^{\cU, \cV}_{1,0}$ along $[0,1)$ and along $[1,1+\veps)$.}
  \label{fig:cubical-interleaving-ss}
\end{figure}
\end{exmp}

\subsection{Interpolating covers and spectral sequence interleavings}
Consider $X \in \FCWcpx$, together with a pair of covers $\cW$ and $\cU$ so that $\cW \prec \cU$. Motivated by the interleaving constructed in  Thm.~\ref{thm:inverse-refinement} we take a closer look at the following finite sequence of covers interpolating between $\cW$ and a cover that both refines and is refined by $\cU$. Let the strict $r$-th intersections of $\cU$ be the family of sets
 $\cU^r = \{\cU_\tau\}_{\tau \in N_\cU^r}$ for all $r\geq 0$.
We define the $(r,\cW,\cU)$-\emph{interpolation} as the covering set
$\cW^r = \cW \cup \cU^r$. In particular, note that the $(0,\cW,\cU)$-interpolation
has the property that $\cW^0 \prec \cU \prec \cW^0$, and consequently
$E^2_{p,q}(X, \cU) \simeq E^2_{p,q}(X, \cW^0)$. In addition if $\cU$ is a
finite cover, then we will have $\cU^N = \emptyset$ for $N\geq 0$ sufficiently large and
consequently $\cW^N = \cW$.

\begin{prop}[Local Checks]
\label{prop:local-checks}
Let $\cW\prec \cU$ be a pair of covers for $X$, where $\cU$ is finite. Let $N \geq 0$ be such that $\cU^N = \emptyset$. For every $0 \leq r \leq N$, we assume that there
exist $\veps_r \geq 0$ and $\nu_r \geq 0$ such that for
all $\tau \in N_{\cU}^r$
\[
E^2_{0,q}\big(\cU_\tau, \cW^{r+1}_{|\cU_\tau}\big) \sim_R^{\nu_r}
E^\infty_{0,q} \big(\cU_\tau, \cW^{r+1}_{|\cU_\tau}\big) \sim_L^{\nu_r}
\PH_q(\cU_\tau)
\]
and also
\[
d_I(E^2_{p,q}(\cU_\tau,\cW^{r+1}_{|\cU_\tau}),0) \leq \veps_r\ .
\]
for all $p>0$, and $q\geq 0$. Then we have that
\[
d_I^2\big(E^*_{p,q}(X, \cW^k), E^*_{p,q}(X, \cW^{k+1})\big)
\leq 2\max(\veps_r, \nu_r).
\]
Therefore, by using the triangle inequality, we obtain
$$
d_I^2\big(E^*_{p,q}(X, \cU), E^*_{p,q}(X, \cW)) \leq \sum_{k=0}^N 2\max(\veps_r,\nu_r)\ .
$$
\end{prop}
\begin{proof}
We need to consider the spectral sequence $^{\rm II}E_{p,q}^2(\cW^{r+1},\cW^r;\PH_k)$. Note that by construction of the covers $\cW^r$ we have that for each $\tau \in N_{\cU}^r$ with $\dim(\tau) > 0$ the set $\cW^{r}_\tau$ is contained in one of the open sets from $\cW^{r+1}$. By Lemma~\ref{lem:cover-containing-space} this implies that $^{\rm II}E_{p,q}^1(\cW^{r+1},\cW^r;\PH_k) = 0$ for all $p>0$ and
$q>0$ and $k \geq 0$. Moreover, we have that ${^{\rm II}E}_{0,q}^1(\cW^{r+1}, \cW^{r};\PH_k) = \bigoplus_{\tau \in N^q_{\cW^r}} \PH_k(\cW^{r}_\tau)$ for all $q>0$ and $k\geq 0$.
The resulting spectral sequence is shown in Fig.~\ref{fig:interpolation-first-page}.

As a consequence of these observations condition~(\ref{eq:RL-interleavings-M}) holds for these indices with $\nu=0$.
In addition,
${^{\rm II}E}_{0,q}^2(\cW^{r+1}, \cW^{r};\PH_k) = E_{q,k}^2(X, \cW^r)$
holds for all $q \geq 2$ and $k \geq 0$ (see Fig.~\ref{fig:interpolation-first-page}
and~\ref{fig:interpolation-second-page}).
In particular, there is only one possible non-trivial differential for each entry in the bottom row as indicated in Fig.~\ref{fig:interpolation-second-page}. Note that our hypothesis $d_I(E^2_{p,q}(\cU_\tau,\cW^{r+1}_{|\cU_\tau}),0) \leq \veps_r$ applies to the entries in the first column with $p>0$ and gives left and right interleavings of the form
$$
\check{\cH}_q(\cW^{r+1}\cap\cW^r;\PH_k) \sim_R^{\veps_r}
{^{\rm II}E}_{0,q}^\infty (\cW^{r+1}, \cW^r;\PH_k) \sim_L^{\veps_r}
 {^{\rm II}E}_{0,q}^2 (\cW^{r+1}, \cW^r;\PH_k)
$$
for all $q> 0$ and $k \geq 0$. Hence, condition~(\ref{eq:RL-interleavings-IIE})
holds with value $\veps_r$.

\begin{figure}[ht]
$$
\begin{tikzcd}[every matrix/.append style={nodes={font=\tiny}},column sep=1.1cm]
{^{\rm II}E}_{2,0}^1 (\cW^{r+1}, \cW^r;\PH_k) & 0 &  0 & \ddots \\
{^{\rm II}E}_{1,0}^1 (\cW^{r+1}, \cW^r;\PH_k) & 0 & 0 & 0\\
{^{\rm II}E}_{0,0}^1 (\cW^{r+1}, \cW^r;\PH_k) &
\bigoplus \limits_{\tau \in N_{\cW^{r}}^1} \PH_k(\cW^{r}_\tau)
\ar[l, "d_1", swap] &
\bigoplus \limits_{\tau \in N_{\cW^{r}}^2} \PH_k(\cW^{r}_\tau) \ar[l] &
\bigoplus \limits_{\tau \in N_{\cW^{r}}^3} \PH_k(\cW^{r}_\tau) \ar[l]
\end{tikzcd}
$$
\caption{First page of ${^{\rm II}E}_{p,q}^* (\cW^{r+1}, \cW^r;\PH_k)$.}
\label{fig:interpolation-first-page}
\end{figure}
\begin{figure}[ht]
{ \small
$$
\begin{tikzcd}[column sep=0.5em]
\sim \veps_r & 0 & 0 & \ddots \\
\sim \veps_r & 0 & 0 & 0 \\
{^{\rm II}E}_{0,0}^2 (\cW^{r+1}, \cW^r;\PH_k) &
{^{\rm II}E}_{0,1}^2 (\cW^{r+1}, \cW^r;\PH_k) &
E_{2,k}^2(X, \cW^r) \ar[llu, "d_2", swap] &
E_{3,k}^2(X, \cW^r) \ar[llluu, "d_3", swap]
\end{tikzcd}
$$}
\caption{Second page of ${^{\rm II}E}_{p,q}^* (\cW^{r+1}, \cW^r;\PH_k)$ together with
higher differentials.}
\label{fig:interpolation-second-page}
\end{figure}

Let us look now at the case $q=0$. Here we have
$\check{\cH}_0(\cW^{r+1}\cap\cW^r;\PH_k) = {^{\rm II}E}^2_{0,0}(\cW^{r+1}, \cW^r; \PH_k)$ and
consequently~(\ref{eq:RL-interleavings-IIE}) holds with value $\veps =0$.
Next, by hypothesis, for all $k\geq 0$ we have right and left interleavings
$$
M^2_{0,k}\big(\cU_\tau \cap \cW^{r+1}\big) \sim_R^{\nu_r}
M^\infty_{0,k} \big(\cU_\tau \cap\cW^{r+1}\big) \sim_L^{\nu_r}
\PH_k(\cU_\tau)\ ,
$$
for all $\tau \in N_{\cU}^r$. Thus by taking the direct sum of these interleavings we obtain
$$
^{\rm II}E_{0,0}^1(\cW^{r+1},\cW^r;\PH_k) \sim_R^{\nu_r}
\bigoplus_{\tau \in N^{\cW^r}_0}
M^\infty_{0,k}(\cW^r_\tau \cap \cW^{r+1}) \sim_L^{\nu_r}
E^1_{0,k}(X, \cW^{r})\ .
$$
and condition~(\ref{eq:RL-interleavings-M}) also holds for $q = 0$.
The result now follows from Thm.~\ref{thm:inverse-refinement}.

Notice that we can slightly improve the statement of Thm.~\ref{thm:inverse-refinement} here: For each term in the bottom row of the spectral sequence in this particular example only one of the two conditions~(\ref{eq:RL-interleavings-IIE}) and~(\ref{eq:RL-interleavings-M}) is nontrivial, and the proof of Thm.~\ref{thm:inverse-refinement} carries over with $2\max(\veps_r, \nu_r)$ replacing $2(\veps_r + \nu_r)$.
\end{proof}

\begin{rem}
Notice that for reasonable cases the parameters $\nu_r$ are bounded above by
$K\veps_r$ for some constant $K>0$ by a result from~\citep{GovSkr2018}. Nevertheless, we would like to keep $\nu_r$ and $\veps_r$ separated here, since
we hope to compute it from $M^*_{p,k}\big(\cU_\tau, \cW^{r+1}_{|\cU_\tau}\big)$ for
$\tau \in N_\cU^r$ hereby get more accurate estimates. Intuitively, asking for $\veps_r$ and $\nu_r$ to be small is equivalent to asking for cycle representatives in covers from $\cW^{r}$ to be approximately contained in covering sets from $\cW^{r+1}$.
\end{rem}

Finally, we would like to compare two separate covers $\cU$ and $\cV$ and have an estimate for the interleaving distance between the associated spectral sequences.
The main idea of Prop.~\ref{prop:local-checks} is to
translate this comparison problem into a few local checks that can be run in parallel.
We formalize this in the following Corollary.
\begin{cor}[Stability of Covers]
\label{cor:stability-covers}
Consider two pairs $(X,\cU)$ and $(X,\cV)$, where $X$ is a space and $\cU$ and $\cV$ are covers. Let $\cW = \cU \cap \cV$ and denote by $\cW^r_\cU$ and $\cW^r_\cV$ the
respective $(r,\cW, \cU)$ and $(r,\cW, \cV)$ interpolations.
For every $0 \leq r \leq N$, we assume that there
exist $\veps_r , \veps_r'\geq 0$ and $\nu_r, \nu_r' \geq 0$ such that for
all $\tau \in N_{\cU}^r$ and $\sigma \in N_{\cV}^r$
\begin{gather*}
E^2_{0,q}\big(\cU_\tau, \cW^{r+1}_{\cU}\big) \sim_R^{\nu_r}
E^\infty_{0,q} \big(\cU_\tau, \cW^{r+1}_{\cU}\big) \sim_L^{\nu_r}
\PH_q(\cU_\tau),\\
E^2_{0,q}\big(\cV_\sigma, \cW^{r+1}_{\cV}\big) \sim_R^{\nu_r'}
E^\infty_{0,q} \big(\cV_\sigma, \cW^{r+1}_{\cV}\big)
\sim_L^{\nu_r'}
\PH_q(\cV_\sigma),
\end{gather*}
for all $r \geq 0$, and also
$$
d_I(E^2_{p,q}(\cU_\tau,\cW^{r+1}_{\cU}),0) \leq \veps_r
\qquad, \qquad
d_I(E^2_{p,q}(\cV_\sigma,\cW^{r+1}_{\cV}),0) \leq \veps_r'
$$
for all $p>0$, and $q\geq 0$. Then we have that
$$
d_I^2\big(E^*_{p,q}(X, \cU), E^*_{p,q}(X, \cV)) \leq R(\cU,\cV)
$$
where $R(\cU,\cV)= \max\Big(\sum_{r=0}^N 2\max(\veps_r,\nu_r), \sum_{r=0}^N 2\max(\veps_r',\nu_r')\Big)$.
\end{cor}
\begin{proof}
By Lemma~\ref{lem:refinement-double-complexes} there are double complex morphisms given by the refinement maps
$$
\begin{tikzcd}[column sep=1.5cm]
\check{\cC}_p(\cU, C^\cell_q) &
\check{\cC}_p(\cW, C^\cell_q) \ar[l, swap, "\rho^{\cU, \cW}_{p,q}"] \ar[r, "\rho^{\cV, \cW}_{p,q}"] &
\check{\cC}_p(\cV, C^\cell_q)\ .
\end{tikzcd}
$$
In turn, these induce $2$-morphisms of spectral sequences
$$
\begin{tikzcd}[column sep=1.5cm]
E^2_{p,q}(X,\cU) &
E^2_{p,q}(X,\cW) \ar[l, swap, "\rho^{\cU, \cW}_{p,q}"] \ar[r, "\rho^{\cV, \cW}_{p,q}"] &
E^2_{p,q}(X,\cV)\ .
\end{tikzcd}
$$
Let $\psi_{p,q}^{\cU, \cW}$ and $\psi_{p,q}^{\cV, \cW}$ be the `inverses' of $\rho_{p,q}^{\cU, \cW}$ and $\rho_{p,q}^{\cV, \cW}$, respectively, witnessing the interleavings of the two spectral sequences (see Thm.~\ref{thm:inverse-refinement} and Prop.~\ref{prop:local-checks}). The result follows from considering the commutative diagram
$$
\begin{tikzcd}[row sep=1.5cm]
E^2_{p,q}(X,\cU)  \ar[d, "\Sigma^{R(\cV,\cU)}"] \ar[dr, "\psi^{\cW, \cU}_{p,q}"] &
E^2_{p,q}(X,\cW)\ar[l, swap, "\rho^{\cU, \cW}_{p,q}"] \ar[r, "\rho^{\cV, \cW}_{p,q}"] \ar[d, "\Sigma^{R(\cV,\cU)}"] &
E^2_{p,q}(X,\cV) \ar[d, "\Sigma^{R(\cV,\cU)}"] \ar[dl, "\psi^{\cW, \cV}_{p,q}"', near start, pos=0.4]
\\
E^2_{p,q}(X,\cU)[R(\cV,\cU)]   &
E^2_{p,q}(X, \cW)[R(\cV,\cU)] \ar[l, swap, "\rho^{\cU, \cW}_{p,q}"] \ar[r, "\rho^{\cV, \cW}_{p,q}"] &
E^2_{p,q}(X,\cV)[R(\cV,\cU)]
\end{tikzcd}
$$
where all arrows are $2$-morphisms of spectral sequences.
\end{proof}

\section{Conclusion}
We have introduced spectral sequences associated to geometric realisations of diagrams of CW-complexes and have given examples of such diagrams that are relevant in topological data analysis. We expect them to have a natural use in the distributed computation of persistent homology. We studied spectral sequences as an invariant in their own right in particular their stability properties. To achieve this, we introduced $\veps$-acyclic carriers and equivalences as well as suitable compatibility conditions for diagrams which lead to the stability of the spectral sequences and their targets. We also adapted the tools in~\citep{Serre1955} to study stability with respect to different covers. In particular, given a refinement for a cover, together with its induced map, we have built an inverse morphism up to some shift in the persistence values. Then we used this result to construct interleavings between the second pages of two spectral sequences associated to two different covers on a fixed filtered complex. We hope these tools motivate the further study and use of spectral sequences in applied topology.

\bibliography{library}
\bibliographystyle{abbrv}
\end{document}